\documentclass[reqno]{amsart}

\usepackage{amssymb,amsthm,amsmath,tikz,hyperref,array}
\usepackage[enableskew]{youngtab}
\usepackage{graphics}

\usetikzlibrary{shapes}

\DeclareSymbolFont{symbolsC}{U}{txsyc}{m}{n}
\SetSymbolFont{symbolsC}{bold}{U}{txsyc}{bx}{n}
\DeclareFontSubstitution{U}{txsyc}{m}{n}
\DeclareMathSymbol{\Nearrow}{\mathrel}{symbolsC}{116}
\DeclareMathSymbol{\Searrow}{\mathrel}{symbolsC}{117}
\DeclareMathSymbol{\Nwarrow}{\mathrel}{symbolsC}{118}
\DeclareMathSymbol{\Swarrow}{\mathrel}{symbolsC}{119}

\newtheorem{theorem}{Theorem}[section]
\newtheorem{proposition}[theorem]{Proposition}
\newtheorem{corollary}[theorem]{Corollary}
\newtheorem{lemma}[theorem]{Lemma}
\newtheorem{conjecture}[theorem]{Conjecture}
\newtheorem{question}[theorem]{Question}

{\theoremstyle{definition}
\newtheorem{definition}[theorem]{Definition}
\newtheorem{remark}[theorem]{Remark}
\newtheorem{example}[theorem]{Example}
}

\numberwithin{equation}{section}
\numberwithin{figure}{section}
\numberwithin{table}{section}
 
\renewcommand{\theenumi}{\alph{enumi}}

\newcommand{\domleq}{\preceq}   
\newcommand{\domgeq}{\succeq}   
\newcommand{\rows}[1]{\mathrm{rows}(#1)}
\newcommand{\cols}[1]{\mathrm{cols}(#1)}
\newcommand{\rowsk}[2]{\mathrm{rows}_{#1}(#2)}
\newcommand{\colsk}[2]{\mathrm{cols}_{#1}(#2)}
\newcommand{\rects}[3]{\mathrm{rects}_{#1,#2}(#3)}
\newcommand{\trim}[2]{\mathrm{trim}^{#1}(#2)}

\newcommand{\overlap}[2]{\mathrm{overlap}_{#1}(#2)}
\newcommand{\fsupp}[1]{\mathrm{supp}_F(#1)}
\newcommand{\ssupp}[1]{\mathrm{supp}_s(#1)}
\newcommand{\osupp}[2]{\mathrm{supp}_{#1}(#2)}
\newcommand{\suppfn}[1]{{\mathrm{Fsupp}(#1)}}
\newcommand{\ncn}[1]{{\mathrm{Overlaps}(#1)}}
\newcommand{\unsortedrows}[1]{\mathrm{ur}(#1)}
\newcommand{\unsortedcols}[1]{\mathrm{uc}(#1)}
\newcommand{\comp}[1]{\mathrm{comp}(#1)}

\newcommand{\QS}{S}

\renewcommand{\emptyset}{\varnothing}

\title[Comparing skew Schur functions]{Comparing skew Schur functions: \\a quasisymmetric perspective}

\author{Peter R.\,W. McNamara}
\address{Department of Mathematics, Bucknell University, Lewisburg, PA 17837, USA}
\email{\href{mailto:peter.mcnamara@bucknell.edu}{peter.mcnamara@bucknell.edu}}
\urladdr{\href{http://www.facstaff.bucknell.edu/pm040}{www.facstaff.bucknell.edu/pm040}}
\thanks{This work was partially support by a grant from the Simons Foundation (\#245597 to Peter McNamara)}
 

\begin{document}

\begin{abstract}
Reiner, Shaw and van Willigenburg showed that if two skew Schur functions $s_A$ and $s_B$ are equal, then the skew shapes $A$ and $B$ must have the same ``row overlap partitions.''  
Here we show that these row overlap equalities are also implied by a much weaker condition than Schur equality:  that $s_A$ and $s_B$ have the same support when expanded in the fundamental quasisymmetric basis $F$.  Surprisingly, there is significant evidence supporting a conjecture that the converse is also true.

In fact, we work in terms of inequalities, showing that if the $F$-support of $s_A$ contains that of $s_B$, then the row overlap partitions of $A$ are dominated by those of $B$, and again conjecture that the converse also holds.  Our evidence in favor of these conjectures includes their consistency with a complete determination of all $F$-support containment relations for $F$-multiplicity-free skew Schur functions.  We conclude with a consideration of how some other quasisymmetric bases fit into our framework.  
\end{abstract}

\maketitle

\tableofcontents

\section{Introduction}

For well-documented reasons (see, for example, \cite{Ful97, Ful00, Sag01, ec2}), the Schur functions $s_\lambda$ are often considered to be the most important basis for symmetric functions.  Furthermore, skew Schur functions $s_{\lambda/\mu}$ are both a natural generalization of Schur functions and a fundamental example of \emph{Schur-positive} functions, meaning that when expanded in the basis of Schur functions, all the coefficients are nonnegative.  The coefficients that result are the Littlewood--Richardson coefficients, which also arise in the representation theory of the symmetric and general linear groups, in the study of the cohomology ring of the Grassmannian, and in certain problems about eigenvalues of Hermitian matrices.  More information on these connections can be found in the aforementioned references.  

For skew shapes $A$ and $B$, determining conditions for the expression 
\begin{equation}\label{equ:difference}
s_A -s_B
\end{equation}
to be Schur-positive is a problem that has received much attention in recent years.  See, for example, \cite{BBR06, FFLP05, KWvW08, Kir04, LPP07, LLT97, McN08, McvW09b, McvW12, Oko97}.  It is well known that this question is currently intractable when stated in anything close to full generality.  A weaker condition than $s_A-s_B$ being Schur-positive is that the Schur support of $s_B$ is contained in the Schur support of $s_A$.  The \emph{Schur support} of $s_A$, also called the Schur support of $A$ and denoted $\ssupp{A}$, is defined to be the set of those $\lambda$ for which $s_\lambda$ appears with nonzero coefficient when we expand $s_A$ as a linear combination of Schur functions.  Support containment for skew Schur functions is directly relevant to the results of \cite{DoPy07, FFLP05, McN08, McvW09b, McvW12}; let us give the flavor of just one beautiful result about the support of skew Schur functions.  There exist Hermitian matrices $A$, $B$ and $C=A+B$, with eigenvalue sets $\mu$, $\nu$ and $\lambda$ respectively, if and only if $\nu$ is in the Schur support of $s_{\lambda/\mu}$.  (See the survey \cite{Ful00} and the references therein.)  

Of the aforementioned papers, the most relevant to the present work is \cite{McN08}, which gives necessary conditions on $A$ and $B$ for $s_A - s_B$ to be Schur-positive or, more generally, for the Schur support of $A$ to contain that of $B$.  These conditions are in terms of dominance order on $\rowsk{k}{A}$, which are partitions first defined in \cite{RSvW07} and which count certain overlaps among the rows of $A$.  We will put our new results in context below by comparing them with the results of \cite{McN08}.

Our goal is to further our understanding of the expression \eqref{equ:difference} and the $\rowsk{k}{A}$ conditions by moving to the setting of quasisymmetric functions.  One starting point for information on the importance and many applications of quasisymmetric functions is \cite{Wik13} and the references therein. 
We will place particular emphasis on the expansion of skew Schur functions in terms of Gessel's basis of fundamental quasisymmetric functions \cite{Ges84}, whose elements we denote by $F_\alpha$ for a composition $\alpha$.  Gessel's original applications of the $F$-basis were in studying $P$-partitions of posets and in enumerating certain permutations.  Like Schur functions, the $F_\alpha$ have a representation-theoretic significance, arising as the characteristics of the irreducible characters of the (type $A$) 0-Hecke algebra \cite{DKLT96, KrTh97}. 

By working in terms of the $F$-basis, we are able to make the advances listed in (a)--(e) below.  The concepts of $F$-positivity and $F$-support are defined, as one would expect, by considering expansions of skew Schur functions in terms of the $F$-basis instead of the Schur basis.   As shown by \cite[Theorem~7.19.7]{ec2} which appears as Theorem~\ref{thm:fexpansion} below, Schur functions are examples of $F$-positive functions.  The diagram shown in Figure~\ref{fig:implications} summarizes implications that are central to this paper.  The first two horizontal arrows are by definition of support, while the diagonal arrows are due to Schur functions being $F$-positive.  The rightmost arrow is our main result, Theorem~\ref{thm:fsupport}.  That this arrow could be replaced by the symbol $\Longleftrightarrow$ is Conjecture~\ref{con:fsupport}.
Before giving more details, let us give examples which will be relevant to the discussion that follows.  
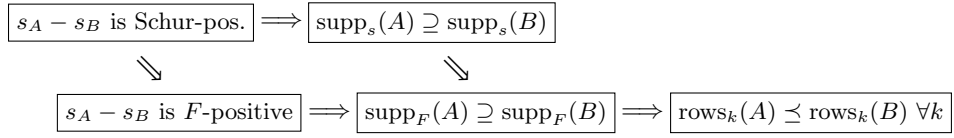
\begin{figure}
\begin{center}
\begin{tikzpicture}[scale=0.8]
\tikzstyle{every node}=[draw, inner sep=3pt]; 
\draw (0,1.5) node {\small$s_A - s_B$ is Schur-pos.};
\draw (0.8,0) node {\small$s_A - s_B$ is $F$-positive};
\draw (5,1.5) node {\small$\ssupp{A} \supseteq \ssupp{B}$};
\draw (5.9,0) node {\small$\fsupp{A} \supseteq \fsupp{B}$};
\draw (11.3,0) node {\small$\rowsk{k}{A} \domleq \rowsk{k}{B}\ \forall k$};
\tikzstyle{every node}=[]; 
\draw (2.5,1.5) node {$\Longrightarrow$};
\draw (3.3,0) node {$\Longrightarrow$};
\draw (8.5,0) node {$\Longrightarrow$};
\draw (0.3,0.75) node {$\Searrow$};
\draw (5.4,0.75) node {$\Searrow$};
\end{tikzpicture}
\caption{A summary of the implications most pertinent to this paper.  Here and elsewhere, $A$ and $B$ are skew shapes, and $\ssupp{A}$ (respectively $\fsupp{A}$) denotes the Schur support (resp.\ $F$-support) of $A$.}
\label{fig:implications}
\end{center}
\end{figure}
\begin{example}\label{exa:intro}
The three skew shapes shown here tend to be useful for providing counterexamples.
\[
\begin{array}{cccccc}
& 
\begin{tikzpicture}[scale=0.4]
\begin{scope}
\draw (-1.5,1.5) node {$A_1=$};
\draw[thick] (0,0) -- (1,0) -- (1,2) -- (3,2) -- (3,3) -- (1,3) -- (1,2) -- (0,2) -- cycle;
\draw (2,3) -- (2,2);
\draw (0,1) -- (1,1);
\end{scope}
\end{tikzpicture}
&&
\begin{tikzpicture}[scale=0.4]
\begin{scope}
\draw (-1.5,1.5) node {$A_2=$};
\draw[thick] (0,0) -- (1,0) -- (1,1) -- (2,1) -- (2,2) -- (3,2) -- (3,3) -- (1,3) -- (1,1) -- (0,1) -- cycle;
\draw (2,3) -- (2,2) -- (1,2);
\end{scope}
\end{tikzpicture}
&&
\begin{tikzpicture}[scale=0.4]
\begin{scope}
\draw (-1.5,1) node {$A_3=$};
\draw[thick] (0,0) -- (2,0) -- (2,2) -- (0,2) -- cycle;
\draw (1,0) -- (1,2);
\draw (0,1) -- (2,1);
\end{scope}
\end{tikzpicture}
\medskip \\ 
\mbox{Schur}
& 
s_{31} + s_{211}
&&
s_{31} + s_{22} + s_{211}
&&
s_{22}
\\ 
\mbox{expansion}
\smallskip \\
\mbox{$F$-expansion}
& 
F_{31} + F_{13} + F_{22} + {}
&&
F_{31} + F_{13} + 2F_{22} + {} 
&&
F_{22} + F_{121}
\\
& 
F_{211} + F_{121} + F_{112} 
&&
F_{211} + 2F_{121} + F_{112}
\end{array}
\]
\end{example}

As promised, here are the full details of our advances.  
\begin{enumerate}
\item\label{ite:equality}  It is shown in \cite{RSvW07} that if $s_A = s_B$ for skew shapes $A$ and $B$, then $A$ and $B$ have equal sets of row overlap partitions.  This result was strengthened in \cite{McN08} by showing that the same conclusion holds under the weaker assumption that the Schur supports of $A$ and $B$ are equal.  We show in Corollary~\ref{cor:equality} that the $F$-supports of $A$ and $B$ being equal is enough to imply $A$ and $B$ have equal sets of row overlap partitions.  This is a strengthening of the result from \cite{McN08} for the following reasons: if the Schur supports of $A$ and $B$ are equal, then it follows from \cite[Theorem~7.19.7]{ec2} that their $F$-supports are equal.  However, the converse is not true, as shown by $A_1$ and $A_2$ of Example~\ref{exa:intro}.  
\item\label{ite:fpositive} In a similar vein, it is shown in \cite{McN08} that if $s_A-s_B$ is Schur-positive, then the row overlap partitions of $A$ are dominated by those of $B$.  We show in Corollary~\ref{cor:fpositive} that the same conclusion can be drawn under the weaker assumption that $s_A - s_B$ is $F$-positive.  Referring to Example~\ref{exa:intro}, consider $s_{A_1} - s_{A_3}$ for an expression that is $F$-positive but not Schur-positive.   
\item\label{ite:fsupport} The two previous advances both follow from the following stronger new result in terms of supports.  It is shown in \cite{McN08} that if the Schur support of $A$ contains that of $B$, then the row overlap partitions of $A$ are dominated by those of $B$.  We prove in Theorem~\ref{thm:fsupport} that the same conclusion can be drawn under the weaker assumption that the $F$-support of $A$ contains that of $B$.  Again, $A_1$ and $A_3$ serve as an example. 

As an application, the contrapositive of Theorem~\ref{thm:fsupport} gives a very simple way to show that the $F$-support of $A$ does not contain the $F$-support of $B$, which implies, among other things, that $s_A-s_B$ is not Schur-positive.  
 
\item\label{ite:conjecture} As shown by $A_1$ and $A_3$ of Example~\ref{exa:intro}, it is certainly not the case that if the row overlaps of $A$ are dominated by those of $B$, then the Schur support of $A$ contains that of $B$.  However, we offer Conjecture~\ref{con:fsupport}: if $A$ and $B$ have the same number of boxes, then the row overlaps of $A$ are dominated by those of $B$ if and only if the $F$-support of $A$ contains that of $B$.  As a result, examining the rows overlaps would give a quick way to determine containment of $F$-supports.  In terminology we will define, the conjecture implies that the $F$-support poset is isomorphic to the overlaps poset. Therefore, the conjecture asserts that $F$-support containment somehow encapsulates exactly the relationship implied by dominance of row overlap partitions.  
Cases for which Conjecture~\ref{con:fsupport} holds include ribbons whose rows all have length at least 2, and all skew shapes with at most 12 boxes.

\item\label{ite:multfree} Bessenrodt and van Willigenburg \cite{BevW13} have classified all those skew shapes $A$ that are $F$-multiplicity-free, i.e., when $s_A$ is expanded in the $F$-basis, all coefficients are 0 or 1.  In Theorem~\ref{thm:multfree}, we determine completely the $F$-positivity and $F$-support comparabilities among $F$-multiplicity-free skew shapes.  The analogous relationships for the Schur multiplicity-free skew shapes are only known in special cases (for example, see \cite{McvW09b} for Schur multiplicity-free ribbons).  We then show that these $F$-support comparabilities are exactly as predicted by Conjecture~\ref{con:fsupport}.  
\end{enumerate}

We conclude with a consideration of other quasisymmetric function bases, specifically the monomial quasisymmetric functions, the quasisymmetric Schur functions of Haglund et al.\ \cite{HLMvW11}, and the dual immaculate basis of Berg et al.\ \cite{BBSSZpr}.  We augment Figure~\ref{fig:implications} by determining the positivity and support-containment implications involving these bases (see Figure~\ref{fig:moreimplications}).

The rest of the paper is organized as follows.  We give the preliminaries and relevant prior results in Sections~\ref{sec:prelims} and~\ref{sec:priorresults}, respectively.  Result \eqref{ite:fsupport} above and its consequences \eqref{ite:equality} and \eqref{ite:fpositive} are the topic of Section~\ref{sec:onedirection}.  In Section~\ref{sec:mainconjecture}, we present the converse conjecture (Conjecture~\ref{con:fsupport}) and offer evidence in its favor.  Section~\ref{sec:multfree} contains the results from \eqref{ite:multfree} about $F$-multiplicity-free skew shapes. We conclude in Section~\ref{sec:conclusion} with a consideration of how other quasisymmetric function bases fit into our framework.

\section{Preliminaries}\label{sec:prelims}

\subsection{Compositions, partitions and skew shapes}

Given a nonnegative integer $n$, a \emph{composition} of $n$ is a sequence $\alpha$ of positive integers whose sum is $n$.  We call $n$ the \emph{size} of $\alpha$ and denote it $|\alpha|$.  If $\alpha$ is weakly decreasing then it is said to be a \emph{partition} of $n$.  Let $\emptyset$ denote the unique partition of 0. 

We will follow the custom of letting $[n]$ denote the set $\{1,\ldots,n\}$.  For fixed $n$, there is a well-known bijection from compositions $\alpha = (\alpha_1, \ldots, \alpha_k)$ of $n$ to subsets  of $[n-1]$ that sends $\alpha$ to the set $S(\alpha)$ defined by
\[
S(\alpha) = \{\alpha_1, \alpha_1+\alpha_2, \ldots, \alpha_1+\alpha_2+\cdots+\alpha_{k-1}\}.
\]
If $S(\alpha)=T$, then we say $\alpha$ is the composition corresponding to the set $T$, and write $\comp{T}=\alpha$ for the inverse map.

Given a partition $\lambda$, we define its \emph{Young diagram} to be a left-justified array of boxes with $\lambda_i$ boxes in the $i$th row from the top.  If the Young diagram of another partition $\mu$ is contained in that of $\mu$, then the \emph{skew shape} $\lambda/\mu$ is obtained by removing the boxes corresponding to $\mu$ from the top-left of the Young diagram of $\lambda$. 
For example, the skew shapes from Example 1.1 can be expressed as $311/1$, $321/11$ and $22/\emptyset = 22$, respectively.  We will typically refer to skew shapes using uppercase Roman letters.  The \emph{size} of a skew shape $A$ is its number of boxes and is denoted $|A|$.

A \emph{horizontal strip} is a skew shape that has at most one box in each column, with \emph{vertical strips} defined similarly.  The \emph{transpose} $\lambda^t$ of a partition $\lambda$ is the partition obtained by reading the column lengths of the Young diagram of $\lambda$ from left to right.  For example $(443)^t = 3332$.  The transpose of a skew shape $A=\lambda/\mu$ is $A^t = \lambda^t/\mu^t$. 

For a skew shape $A$, let $\rows{A}$ (resp.\ $\cols{A}$) denote the partition consisting of the row (resp.\ column) lengths of $A$ sorted into weakly decreasing order.  A \emph{ribbon} is a skew shape in which every pair of adjacent rows overlap in exactly one column.  In particular, note that a ribbon is completely determined by its row lengths from top to bottom.  This allows us to define the notion of $\rows{\alpha}$ and $\cols{\alpha}$ for a composition $\alpha$ as $\rows{R}$ and $\cols{R}$ respectively, where $R$ is the ribbon whose row lengths from top to bottom are given by $\alpha$.  For example, $\rows{1311} = 3111$ and $\cols{1311} = 321$; in general, $\rows{\alpha}$ simply means the weakly decreasing reordering of the parts of $\alpha$.  Observe that, for example, $\cols{22}=22$ when we consider 22 to be a skew shape whereas $\cols{22}=211$ when we consider 22 to be a composition; we will ensure the meaning of our notation is clear from the context. 

We place a partial order on the set of all partitions according to the following definition.  

\begin{definition}
For partitions $\lambda = (\lambda_1, \lambda_2, \ldots, \lambda_r)$ and 
$\mu = (\mu_1, \mu_2, \ldots, \mu_s)$, we define \emph{dominance order} $\domleq$ by
$\lambda \domleq \mu$ if 
\[
\lambda_1 + \lambda_2 + \cdots \lambda_k \leq \mu_1 + \mu_2 + \cdots \mu_k
\]
for all $k=1,2,\ldots,r$, where we set $\mu_i=0$ if $i > s$.
In this case, we will say that $\mu$ \emph{dominates} $\lambda$, or is \emph{more dominant}
than $\lambda$.
\end{definition}

Note that the above definition makes sense even if $\lambda$ and $\mu$ are partitions of different size, as can be the case later when we compare $\rowsk{k}{A}$ and $\rowsk{k}{B}$ for $k\geq 2$.

As in \cite{McN08}, we will need the following result about this extended definition of dominance order.  Since it is straightforward to check, we leave the proof as an exercise.

\begin{lemma}\label{lem:dom_containment}
Consider two sequences $a = (a_1, \ldots, a_r)$ and $b = (b_1, \ldots, b_s)$ of nonnegative integers such that 
$r \leq s$ and $a_i \leq b_i$ for $i=1,2,\ldots,r$.  Let $\alpha$ and $\beta$ denote the partitions
obtained by sorting the parts of $a$ and $b$ respectively into weakly decreasing order.
Then $\alpha \domleq \beta$.  
\end{lemma}

\subsection{Quasisymmetric functions}

For a formal power series $f$ in the variables $x_1, x_2, \ldots$, let $[x_{i_1}^{a_1} x_{i_2}^{a_2} \cdots x_{i_k}^{a_k}]f$ denote the coefficient of $x_{i_1}^{a_1} x_{i_2}^{a_2} \cdots x_{i_k}^{a_k}$ in the expansion of $f$ into monomials.

\begin{definition}
A quasisymmetric function in the variables $x_1, x_2,\ldots$, say with rational coefficients, is a formal power series $f \in \mathbb{Q}[[x_1, x_2, \ldots]]$ of bounded degree such that for every sequence $a_1, a_2, \ldots a_k$ of positive integers, we have
\[
[x_{i_1}^{a_1} x_{i_2}^{a_2} \cdots x_{i_k}^{a_k}]f
= [x_{j_1}^{a_1} x_{j_2}^{a_2} \cdots x_{j_k}^{a_k}]f
\]
whenever $i_1 < i_2 < \cdots < i_k$ and $j_1 < j_2 < \cdots < j_k$.
\end{definition}
As an example, the formal power series
\[
\sum_{1 \leq i < j} x_i^2 x_j
\]
is quasisymmetric but not symmetric.  

For a composition $\alpha=(\alpha_1, \ldots, \alpha_k)$, we define the \emph{monomial quasisymmetric function} $M_\alpha$ by 
\begin{equation}\label{equ:monomials}
M_\alpha = \sum_{i_1 < \cdots < i_k} x_{i_1}^{\alpha_1} \cdots x_{i_k}^{\alpha_k}.
\end{equation}
It is clear that the set $\{M_\alpha\}$, where $\alpha$ ranges over all compositions of size $n$, is a basis for the vector space of quasisymmetric functions of degree $n$.  A more important basis for our purposes is the basis of \emph{fundamental quasisymmetric functions} $F_\alpha$ defined by 
\begin{equation}\label{equ:mtof}
F_\alpha = \sum_{S(\alpha) \subseteq T \subseteq [n-1]} M_{\comp{T}}
\end{equation}
when $\alpha$ has size $n$.  
For example, $F_{22} = M_{22} + M_{211} + M_{112} + M_{1111}$.

For a skew shape $A$ with $n$ boxes, a \emph{standard Young tableau (SYT)} of shape $A$ is a filling of the boxes of $A$ with the numbers $1, 2, \ldots, n$, each used exactly once, so that the numbers increase down the columns and from left to right along the rows.  For example,
\[
\begin{tikzpicture}[scale=0.4]
\draw[thick] (0,0) -- (0,1) -- (1,1) -- (1,3) -- (3,3) -- (3,1) -- (2,1) -- (2,0) -- cycle;
\draw (1,0) -- (1,1) -- (2,1) -- (2,3);
\draw (1,2) -- (3,2);
\draw (0.5,0.5) node {6};
\draw (1.5,0.5) node {4};
\draw (1.5,1.5) node {3};
\draw (1.5,2.5) node {1};
\draw (2.5,1.5) node {5};
\draw (2.5,2.5) node {2};
\end{tikzpicture}
\]
is an SYT  of shape 332/11.
The \emph{descent set} $S$ of an SYT $T$ of shape $A$ is the set of numbers $i$ for which $i+1$ appears in a lower row than $i$.  The \emph{descent composition} of $T$, denoted $\comp{T}$, is then the composition of $|A|$ corresponding to $S$.  For example, the SYT above has descent set $\{2,3,5\}$ and descent composition 2121.

Since the following result, which appears as \cite[Theorem~7.19.7]{ec2}, expresses skew Schur functions in the $F$-basis, it is crucial to this paper and is the reason why the $F$-basis is a natural choice of quasisymmetric basis when comparing skew Schur functions.  Although skew Schur functions are typically defined as a sum of monomials, Theorem~\ref{thm:fexpansion} can also serve as a definition of skew Schur functions for our purposes.   

\begin{theorem}[\cite{Ges84,StaThesis71,StaThesis}]\label{thm:fexpansion}
For a skew shape $A$, we have
\[
s_A = \sum_T F_{\comp{T}}
\]
where the sum is over all standard Young tableau $T$ of shape $A$.  
\end{theorem}

For example, the SYT above contributes $F_{2121}$ to $s_{332/11}$.

Theorem~\ref{thm:fexpansion} tells us that $s_A$ is an example of an \emph{$F$-positive} symmetric function, meaning that it has all nonnegative coefficients when expanded in the $F$-basis.  Analogously to Schur support, we define the \emph{$F$-support} of $A$, denoted $\fsupp{A}$, to be the set of compositions $\alpha$ such that $F_\alpha$ appears with positive coefficient when $s_A$ is expanded in the $F$-basis. 
For any other quasisymmetric basis $\{B_\alpha\}$, analogous definitions of \emph{$B$-positive} and \emph{$B$-support} are obtained by replacing $F$ with $B$. 

\section{Prior results}\label{sec:priorresults}

In \cite{RSvW07}, Reiner, Shaw and van Willigenburg gave sufficient conditions for two skew shapes to yield the same skew Schur function.  More relevant for the purposes of the current discussion is that they also wrote one section (Section~8) on \emph{necessary} conditions for two skew shapes $A$ and $B$ to satisfy $s_A = s_B$.  Their necessary conditions are dependent on certain overlaps among the rows of a skew shape.  Before discussing their work, let us first state a relevant classical result along the same lines; it can be considered a starting point for necessary conditions for skew Schur equality.  A proof in our terminology can be found in \cite{McN08}, and earlier proofs can be found in \cite{Lam77, Zab}.  

\begin{proposition}\label{pro:extreme_fillings}  Let $A$ and $B$ be skew shapes.  If $\lambda \in \ssupp{A}$, then 
\[
\rows{A} \domleq \lambda \domleq \cols{A}^t, 
\]
 and both $s_{\rows{A}}$ and $s_{\cols{A}^t}$ appear with coefficient 1 in the Schur expansion of $s_A$.  Consequently, if $\ssupp{A} \supseteq \ssupp{B}$, then
\[
\rows{A} \domleq \rows{B}  \mbox{\ \ and \ \ } \cols{A} \domleq \cols{B}.
\]
\end{proposition}
 
Reiner, Shaw and van Willigenburg generalized $\rows{A}$ and $\cols{A}$ using the following key definition.  

\begin{definition}
Let $A$ be a skew shape with $r$ rows.  For $i = 1, \ldots, r-k+1$, define $\overlap{k}{i}$
to be the number of columns occupied in common by rows $i, i+1, \ldots, i+k-1$.  
Then $\rowsk{k}{A}$ is defined to be the weakly decreasing rearrangement of 
\[
(\overlap{k}{1}, \overlap{k}{2}, \ldots, \overlap{k}{r-k+1}).
\]
Similarly, we define $\colsk{k}{A}$ by looking at the overlap among the columns of $A$.  
\end{definition}

In particular, note that $\rowsk{1}{A} = \rows{A}$ and $\colsk{1}{A} = \cols{A}$.  

\begin{example}
Let $A=553111/31$ as shown here.
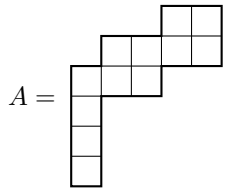
\[
\begin{tikzpicture}[scale=0.4]
\draw (-1.3,3) node {$A=$};
\draw[thick] (0,0) -- (0,4) -- (1,4) -- (1,5) -- (3,5) -- (3,6) -- (5,6) -- (5,4) -- (3,4) -- (3,3) -- (1,3) -- (1,0) -- cycle;
\draw (0,1) -- (1,1);
\draw (0,2) -- (1,2);
\draw (0,3) -- (1,3);
\draw (1,4) -- (3,4);
\draw (3,5) -- (5,5);
\draw (1,3) -- (1,4);
\draw (2,3) -- (2,5);
\draw (3,4) -- (3,5);
\draw (4,4) -- (4,6);
\end{tikzpicture}
\]
We have that 
$\rowsk{1}{A} = 432111$, $\rowsk{2}{A}=22111$, $\rowsk{3}{A}=11$, $\rowsk{4}{A}=1$, and
$\rowsk{i}{A} = \emptyset$ otherwise.  In addition,
$\colsk{1}{A} = 42222$, $\colsk{2}{A}=2211$, $\colsk{3}{A}=111$, $\colsk{4}{A}=1$, and
$\colsk{i}{A} = \emptyset$ otherwise.
\end{example}

It turns out that knowledge of $\rowsk{k}{A}$ for all $k$ is equivalent to knowledge of $\colsk{\ell}{A}$ 
for all $\ell$.  To show this, the natural concept of $\rects{k}{\ell}{A}$ was introduced in \cite{RSvW07}. Here is their result.  

\begin{proposition}[\cite{RSvW07}]\label{pro:rects}
Given a skew shape $A$, consider the doubly-indexed array 
\[
(\rects{k}{\ell}{A})_{k,\ell \geq 1}
\]
where $\rects{k}{\ell}{A}$ is defined to be the number of $k \times \ell$ 
rectangular subdiagrams contained inside
$A$.  Any one of the three forms of data
\[
(\rowsk{k}{A})_{k \geq 1},\ \  (\colsk{\ell}{A})_{\ell \geq 1}, \ \  (\rects{k}{\ell}{A})_{k, \ell \geq 1}
\]
on $A$ determines the other two uniquely.  
\end{proposition}

The main necessary condition from \cite{RSvW07} for skew Schur equality is the following.

\begin{theorem}[\cite{RSvW07}]\label{thm:rsvw}
Let $A$ and $B$ be skew shapes.  If $s_A = s_B$, then
the following three equivalent conditions are true:
\begin{itemize}
\item $\rowsk{k}{A} = \rowsk{k}{B}$ for all $k$;
\item $\colsk{\ell}{A} = \colsk{\ell}{B}$ for all $\ell$;
\item $\rects{k}{\ell}{A} = \rects{k}{\ell}{B}$ for all $k, \ell$.  
\end{itemize}
\end{theorem}

There are two results from \cite{McN08} relevant to this section.  The first extends Proposition~\ref{pro:rects} to the setting of inequalities.

\begin{proposition}\label{pro:rectsineq}
Let $A$ and $B$ be skew shapes.  Then the following conditions are equivalent:
\begin{itemize}
\item $\rowsk{k}{A} \domleq \rowsk{k}{B}$ for all $k$;
\item $\colsk{\ell}{A} \domleq \colsk{\ell}{B}$ for all $\ell$;
\item $\rects{k}{\ell}{A} \leq \rects{k}{\ell}{B}$ for all $k, \ell$.  
\end{itemize}
\end{proposition}

The second result from \cite{McN08} is the corresponding analogue of Theorem~\ref{thm:rsvw}.

\begin{theorem}\label{thm:combine}
Let $A$ and $B$ be skew shapes.
If $s_A - s_B$ is Schur-positive, 
or if $A$ and $B$ satisfy the weaker condition that $\ssupp{A} \supseteq \ssupp{B}$, then
the following three equivalent conditions are true:
\begin{itemize}
\item $\rowsk{k}{A} \domleq \rowsk{k}{B}$ for all $k$;
\item $\colsk{\ell}{A} \domleq \colsk{\ell}{B}$ for all $\ell$;
\item $\rects{k}{\ell}{A} \leq \rects{k}{\ell}{B}$ for all $k, \ell$.  
\end{itemize}
\end{theorem}

A motivation behind \cite{McN08} was to determine easily testable conditions that would show that $s_A - s_B$ is not Schur-positive for certain skew shapes $A$ and $B$.  Theorem~\ref{thm:combine} provides such conditions, as demonstrated by the following example.

\begin{example}\label{exa:incomparable}
Let 
\[
\begin{tikzpicture}[scale=0.4]
\begin{scope}
\draw (-1.5,2) node {$A=$};
\draw[thick] (0,0) -- (0,3) -- (1,3) -- (1,4) -- (4,4) -- (4,3) -- (2,3) -- (2,2) -- (1,2) -- (1,0) -- cycle;
\draw (0,1) -- (1,1);
\draw (0,2) -- (1,2);
\draw (1,2) -- (1,3);
\draw (1,3) -- (2,3);
\draw (2,3) -- (2,4);
\draw (3,3) -- (3,4);
\end{scope}
\begin{scope}[xshift=60mm]
\draw (0,2) node {and};
\end{scope}
\begin{scope}[xshift=95mm]
\draw (-1.5,2) node {$B=$};
\draw[thick] (0,0) -- (0,1) -- (1,1) -- (1,3) -- (3,3) -- (3,4) -- (4,4) -- (4,2) -- (3,2) -- (3,1) -- (1,1) -- (1,0) -- cycle;
\draw (2,1) -- (2,3);
\draw (1,2) -- (3,2) -- (3,3) -- (4,3);
\draw (5,2) node {.};
\end{scope}
\end{tikzpicture}
\]
We see that
$\rowsk{2}{A} = 111$ and $\rowsk{2}{B} = 21$.  Thus we know that $s_B - s_A$ is 
not Schur-positive.  On the other hand, $\rowsk{3}{A} = 1$ while $\rowsk{3}{B} = \emptyset$, 
implying that $s_A - s_B$ is not Schur-positive.  Moreover, we can conclude that
$\ssupp{A}$ and $\ssupp{B}$ are incomparable under containment order.
\end{example}

\section{Main result}\label{sec:onedirection}

Our goal for this section is to state and prove our main result, and deduce relevant corollaries.  We begin immediately with the statement of our main result.  

\begin{theorem}\label{thm:fsupport}
Let $A$ and $B$ be skew shapes.
If $\fsupp{A} \supseteq \fsupp{B}$, then
the following three equivalent conditions are true:
\begin{itemize}
\item $\rowsk{k}{A} \domleq \rowsk{k}{B}$ for all $k$;
\item $\colsk{\ell}{A} \domleq \colsk{\ell}{B}$ for all $\ell$;
\item $\rects{k}{\ell}{A} \leq \rects{k}{\ell}{B}$ for all $k, \ell$.  
\end{itemize}
\end{theorem}

For example, applying this theorem in Example~\ref{exa:incomparable} shows that $\fsupp{A}$ and $\fsupp{B}$ are incomparable with respect to containment.  This is a strictly stronger deduction than being incomparable with respect to Schur support containment (cf.\ $A_1$ and $A_3$ from Example~\ref{exa:intro}.)  Moreover, Theorem~\ref{thm:fsupport} is more than just an incremental improvement of Theorem~\ref{thm:combine} since $\fsupp{A} \supseteq \fsupp{B}$ seems to be ``much closer'' to the overlap conditions than $\ssupp{A} \supseteq \ssupp{B}$.  We will make this assertion precise in Section~\ref{sec:mainconjecture} by giving evidence in favor of our conjecture that the converse of Theorem~\ref{thm:fsupport} is also true.  

\subsection{Consequences of the main result}

We postpone the proof until after we have given some consequences of Theorem~\ref{thm:fsupport}.  If $s_A - s_B$ is $F$-positive, then it is clearly the case that $\fsupp{A} \supseteq \fsupp{B}$, so we get the following corollary.

\begin{corollary}\label{cor:fpositive}
Let $A$ and $B$ be skew shapes.  If $s_A - s_B$ is $F$-positive then
the following three equivalent conditions are true:
\begin{itemize}
\item $\rowsk{k}{A} \domleq \rowsk{k}{B}$ for all $k$;
\item $\colsk{\ell}{A} \domleq \colsk{\ell}{B}$ for all $\ell$;
\item $\rects{k}{\ell}{A} \leq \rects{k}{\ell}{B}$ for all $k, \ell$.  
\end{itemize}
\end{corollary}

To see that Corollary~\ref{cor:fpositive} is not equivalent to Theorem~\ref{thm:fsupport}, let $A=A_1$ and $B=A_2$ from Example~\ref{exa:intro}.  Then the hypothesis of Theorem~\ref{thm:fsupport} holds but that of Corollary~\ref{cor:fpositive} does not. 

Next, by Theorem~\ref{thm:fexpansion}, we get that Theorem~\ref{thm:combine} is simply a consequence of Theorem~\ref{thm:fsupport} and Corollary~\ref{cor:fpositive}.  

The consequence involving equalities can be captured by the following statement, which includes the content of Theorem~\ref{thm:rsvw}.  

\begin{corollary}\label{cor:equality}
Let $A$ and $B$ be skew shapes.  If $s_A = s_B$ or $\ssupp{A}=\ssupp{B}$ or $\fsupp{A}=\fsupp{B}$, then the following three equivalent conditions are true:
\begin{itemize}
\item $\rowsk{k}{A} = \rowsk{k}{B}$ for all $k$;
\item $\colsk{\ell}{A} = \colsk{\ell}{B}$ for all $\ell$;
\item $\rects{k}{\ell}{A} = \rects{k}{\ell}{B}$ for all $k, \ell$.  
\end{itemize}
\end{corollary}

\begin{proof}
If $s_A = s_B$ then we have $\ssupp{A}=\ssupp{B}$ which, by Theorem~\ref{thm:fexpansion}, implies $\fsupp{A}=\fsupp{B}$.  By Theorem~\ref{thm:fsupport}, $\fsupp{A} \supseteq \fsupp{B}$ implies that $\rowsk{k}{A} \domleq \rowsk{k}{B}$ for all $k$.  Similarly, $\rowsk{k}{B} \domleq \rowsk{k}{A}$ for all $k$, and so $\rowsk{k}{A} = \rowsk{k}{B}$ for all $k$.  The remainder of the result now follows from Proposition~\ref{pro:rects}.
\end{proof}

\subsection{Proving the main result}

We now work towards a proof of Theorem~\ref{thm:fsupport}.  The overall approach will be much like that for the proof of \cite[Corollary~3.10]{McN08}, which is our Theorem~\ref{thm:combine}, but the details change because we are now working in the $F$-basis.  For example, the easiest inequality for us to show will be that $\colsk{\ell}{A} \domleq \colsk{\ell}{B}$, whereas the $\mathrm{rows}$ inequality was the one proved in \cite{McN08}.  

While we can determine $\cols{\alpha}$ for a composition $\alpha$ by constructing the relevant ribbon, it will be helpful for Proposition~\ref{pro:f_extreme_fillings}(b) below to have an equivalent way to obtain $\cols{\alpha}$.  

\begin{lemma}\label{lem:cols}
For a ribbon $R$ with $|R|=n$, let $\unsortedrows{R}$ (resp.\ $\unsortedcols{R}$) denote the (unsorted) composition of $n$ given by the row (resp.\ column) lengths of $R$ read from top to bottom (resp.\ right to left).  Then the subsets of $[n-1]$ corresponding to $\unsortedrows{R}$ and $\unsortedcols{R}$ are complements of each other.  

Consequently, for a composition $\alpha$ of $n$, to obtain $\cols{\alpha}$ from $\alpha$ follow this 4-step process: obtain the subset $S(\alpha)$ of $[n-1]$, take the complement $S(\alpha)^c$, then construct the corresponding composition $\comp{S(\alpha)^c}$ of $n$, and sort the result into weakly decreasing order.
\end{lemma}

\begin{proof}
Write the numbers $1, 2, \ldots, n$ in sequence from the top right box of $R$ down to the bottom left.  Every box numbered $i$ for $i<n$ is either the highest-numbered box of its row or of its column, and not both.  It is the highest-numbered box of its row (resp.\ column) if and only if $i$ is an element of the subset of $[n-1]$ corresponding to $\unsortedrows{R}$ (resp.\ $\unsortedcols{R}$).  The first assertion of the lemma follows.

The second assertion follows from the definition of $\cols{\alpha}$.  
\end{proof}

Our Proposition~\ref{pro:extreme_fillings} played a key role in the proofs of \cite{McN08}.  To prove Theorem~\ref{thm:fsupport}, we will need the following quasisymmetric analogue of Proposition~\ref{pro:extreme_fillings}.
Although we only need part (a) in this section, it makes sense to prove parts (a) and (b) together; we need (b) because we will use~\eqref{equ:suppf} in the proof of Theorem~\ref{thm:multfree}.

\begin{proposition}\label{pro:f_extreme_fillings}
Let $A$ and $B$ be skew shapes.  If $\alpha \in \fsupp{A}$ then
\begin{enumerate}
\item 
$\rows{\alpha} \domleq \cols{A}^t$,
\item 
$\cols{\alpha} \domleq \rows{A}^t$,
\end{enumerate}
and both inequalities are sharp.  Consequently, if $\fsupp{A} \supseteq \fsupp{B}$, then
\begin{equation}\label{equ:suppf}
\rows{A} \domleq \rows{B}  \mbox{\ \ and \ \ } \cols{A} \domleq \cols{B}.
\end{equation}
\end{proposition}

See Figure~\ref{fig:extreme_fillings} for examples of SYTx giving equality in (a) and (b).
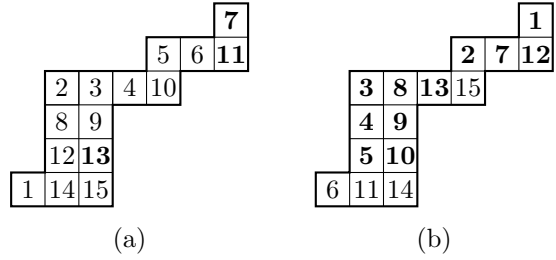
\begin{figure}
\[
\begin{tikzpicture}[scale=0.45]
\begin{scope}
\draw[thick] (0,0) -- (0,1) -- (1,1) -- (1,4) -- (4,4) -- (4,5) --  (6,5) -- (6,6) -- (7,6) -- (7,4) -- (5,4) -- (5,3) -- (3,3) -- (3,0) -- cycle;
\draw (1,0) -- (1,1) -- (3,1);
\draw (2,0) -- (2,4);
\draw (3,3) -- (3,4);
\draw (4,3) -- (4,4) -- (5,4);
\draw (5,4) -- (5,5) -- (7,5);
\draw (6,4) -- (6,5);
\draw (1,3) -- (2,3);
\draw (2,3) -- (2,4);
\draw (3,3) -- (3,4);
\draw (1,2) -- (3,2);
\draw (1,3) -- (3,3);
\draw(0.5,0.5) node {1};
\draw(1.5,0.5) node {14};
\draw(2.5,0.5) node {15};
\draw(1.5,1.5) node {12};
\draw(2.5,1.5) node {\textbf{13}};
\draw(1.5,2.5) node {8};
\draw(2.5,2.5) node {9};
\draw(1.5,3.5) node {2};
\draw(2.5,3.5) node {3};
\draw(3.5,3.5) node {4};
\draw(4.5,3.5) node {10};
\draw(4.5,4.5) node {5};
\draw(5.5,4.5) node {6};
\draw(6.5,4.5) node {\textbf{11}};
\draw(6.5,5.5) node {\textbf{7}};
\draw (3.5,-1) node {(a)};
\end{scope}
\begin{scope}[xshift=90mm]
\draw[thick] (0,0) -- (0,1) -- (1,1) -- (1,4) -- (4,4) -- (4,5) --  (6,5) -- (6,6) -- (7,6) -- (7,4) -- (5,4) -- (5,3) -- (3,3) -- (3,0) -- cycle;
\draw (1,0) -- (1,1) -- (3,1);
\draw (2,0) -- (2,4);
\draw (3,3) -- (3,4);
\draw (4,3) -- (4,4) -- (5,4);
\draw (5,4) -- (5,5) -- (7,5);
\draw (6,4) -- (6,5);
\draw (1,3) -- (2,3);
\draw (2,3) -- (2,4);
\draw (3,3) -- (3,4);
\draw (1,2) -- (3,2);
\draw (1,3) -- (3,3);
\draw(0.5,0.5) node {6};
\draw(1.5,0.5) node {11};
\draw(2.5,0.5) node {14};
\draw(1.5,1.5) node {\textbf{5}};
\draw(2.5,1.5) node {\textbf{10}};
\draw(1.5,2.5) node {\textbf{4}};
\draw(2.5,2.5) node {\textbf{9}};
\draw(1.5,3.5) node {\textbf{3}};
\draw(2.5,3.5) node {\textbf{8}};
\draw(3.5,3.5) node {\textbf{13}};
\draw(4.5,3.5) node {15};
\draw(4.5,4.5) node {\textbf{2}};
\draw(5.5,4.5) node {\textbf{7}};
\draw(6.5,4.5) node {\textbf{12}};
\draw(6.5,5.5) node {\textbf{1}};
\draw (3.5,-1) node {(b)};
\end{scope}
\end{tikzpicture}
\]
\caption{For this skew shape $A$, we have $\rows{A}=433221$, $\rows{A}^t=6531$, $\cols{A}=4422111$ and $\cols{A}^t=7422$.  The descents of the SYTx are shown in bold.  Observe that the SYT in (a) has descent composition $\cols{A}^t$.  The descent composition in (b) is $\alpha=111112111212$, giving $\cols{\alpha}=6531=\rows{A}^t$.}
\label{fig:extreme_fillings}
\end{figure}
\begin{proof}  
By Theorem~\ref{thm:fexpansion}, we know that $\alpha \in \fsupp{A}$ if and only if there exists an SYT $T$ of shape $A$ and descent composition $\alpha$.  First consider (a).  By definition, $\rows{\alpha}_1$ will be the length of the the longest sequence $i, i+1, \ldots, i+j$ such that none of $i, i+1, \ldots, i+j-1$ is a descent in $T$.   Therefore the entries $i, i+1, \ldots, i+j$ appear from left to right in $T$ with no two in the same column. Equivalently, the boxes filled by entries $i, i+1, \ldots, i+j$ form a horizontal strip in $T$, implying that $j+1$ is at most the number of columns of $T$.  In other words, $\rows{\alpha}_1 \leq (\cols{A}^t)_1$.  By the same logic, the elements of the sum $\rows{\alpha}_1 + \cdots +\rows{\alpha}_k$ correspond to a set of $k$ disjoint horizontal strips in $T$.  The number of boxes of any given column of $A$ contained in these $k$ horizontal strips combined is bounded by the minimum of $k$ and the height of the column. 
Compare this with $(\cols{A}^t)_1 + \cdots + (\cols{A}^t)_k$.  Since $(\cols{A}^t)_i$ counts the number of columns of $A$ of height at least $i$, this sum counts the total number of boxes in columns of height less than $k$, plus a contribution of $k$ from each column of height at least $k$.  It follows that
\[
\rows{\alpha}_1 + \cdots +\rows{\alpha}_k  \leq (\cols{A}^t)_1 + \cdots + (\cols{A}^t)_k\,,
\]
as required.

To see that the inequality in (a) is sharp, consider the SYT $T$ of shape $A$ constructed in the following manner.   First, consider the top entry of each nonempty column of $A$, and fill these top entries with $1, 2, \ldots, (\cols{A}^t)_1$ from left to right.  Now consider the skew shape $A^-$ consisting of the boxes that have not yet been filled.  Since $(\cols{A}^t)_k$ counts the number of columns of $A$ of height at least $k$, we know that $A^-$ has $(\cols{A}^t)_2$ columns. Take the top entry of each such column and fill these top entries with $(\cols{A}^t)_1+1, (\cols{A}^t)_1+2, \ldots, (\cols{A}^t)_1 + ( \cols{A}^t)_2$ from left to right.  Continue in this manner until all boxes have been filled.  Because at each stage we filled from left to right and we filled a box in every nonempty column, the descent set of $T$ is  
\[
\{(\cols{A}^t)_1, (\cols{A}^t)_1 + ( \cols{A}^t)_2, \ldots, (\cols{A}^t)_1 +\cdots + ( \cols{A}^t)_{k-1}\},
\] 
where the longest column of $A$ has $k$ boxes.  
In other words, the descent composition $\alpha$ satisfies $\alpha = \rows{\alpha} = \cols{A}^t$, as required. 

The proof of (b) is somewhat similar, except that now we work with vertical strips instead of horizontal strips and fill these vertical strips from top to bottom.   By definition, $\cols{\alpha}_1$ will be the longest sequence $i, i+1, \ldots, i+j$ such that \emph{each} of $i, i+1, \ldots, i+j-1$ is a descent in $T$.   Therefore, the entries $i, i+1, \ldots, i+j$ fill a vertical strip in $A$ from top to bottom, implying that $\cols{\alpha}_1 \leq (\rows{A}^t)_1$.  The rest of the proof is similar to (a).  

To show that the inequality in (b) is sharp, work as in (a) except consider the leftmost entry of each nonempty row instead of the top entry of each column, and fill these leftmost entries from top to bottom.  After completing the filling, the \emph{complement} of the descent set of $T$ in $\{1,2,\ldots,|A|-1\}$ is 
\begin{equation}\label{equ:complement}
\begin{split}
\{(\rows{A}^t)_1, (\rows{A}^t)_1 + ( \rows{A}^t)_2,  \ldots,\\ (\rows{A}^t)_1 +\cdots + ( \rows{A}^t)_{k-1}\},
\end{split}
\end{equation}
where the longest row of $A$ has $k$ boxes.  The composition of $|A|$ corresponding to the set in~\eqref{equ:complement} is $\rows{A}^t$, which has weakly decreasing parts.  By Lemma~\ref{lem:cols}, the descent composition $\alpha$ of $T$ thus satisfies $\cols{\alpha} = \rows{A}^t$, as required.  See Figure~\ref{fig:extreme_fillings}(b) for an example, where the complement of the descent set is $\{6, 11, 14\}$.

The last assertion follows from (a) and (b) and the fact that the transpose operation reverses dominance order when applied to partitions of equal size.
\end{proof}

We need one more concept before giving the proof proper of Theorem~\ref{thm:fsupport}.  For any skew shape $A$, let $\trim{}{A}$ denote the skew shape obtained by deleting the leftmost entry of each nonempty row of $A$.  We will consider $\mathrm{trim}$ to be an operation on skew shapes, meaning that $\trim{\ell}{A} = \trim{}{\trim{\ell-1}{A}}$ and $\trim{1}{A}$ is just $\trim{}{A}$.  
This $\mathrm{trim}$ operation was introduced in \cite{McN08} except there it was defined as deleting the top entry of each nonempty column.

\begin{lemma}\label{lem:trim}
For any skew shape $A$ and $\ell\geq2$, we have
\begin{enumerate}
\item 
$\colsk{\ell-1}{\trim{}{A}} = \colsk{\ell}{A}$;
\item
$
\cols{\trim{\ell-1}{A}} = \colsk{\ell}{A}.
$
\end{enumerate}
\end{lemma}

\begin{proof}
Suppose column $i$ of $A$ contributes $c$ to $\colsk{\ell}{A}$, in the sense that column $i$ of $A$ overlaps with column $i+\ell-1$ in exactly $c$ rows.  We see that this is equivalent to column $i+1$ of $\trim{}{A}$ overlapping with column $i+\ell-1$ in exactly $c$ rows, thus contributing $c$ to $\colsk{\ell-1}{\trim{}{A}}$, implying the result. 

Repeatedly applying (a) gives (b). 
\end{proof}

\begin{proof}[Proof of Theorem~\ref{thm:fsupport}]
We assume that $\fsupp{A} \supseteq \fsupp{B}$ and show that 
\[
\colsk{\ell}{A} \domleq \colsk{\ell}{B} \mbox{\ \ for all $\ell$}.
\]
By Proposition~\ref{pro:rectsineq}, the $\mathrm{rows}$ and $\mathrm{rects}$ conditions will follow.  

With $\ell$ fixed, we will construct a particular SYT $T$ of shape $B$ and descent composition $\alpha$.  
Our choice of $T$ will help us isolate $\cols{\trim{\ell-1}{B}}$  which, by Lemma~\ref{lem:trim}(b), means we will isolate $\colsk{\ell}{B}$.  Roughly speaking, we will start our construction of $T$ so that $\alpha$ is as least dominant as possible, and construct the remainder of $T$ so that $\alpha$ is as dominant as possible.  More precisely,  follow the construction of $T$ from Proposition~\ref{pro:f_extreme_fillings}(b) by considering the leftmost box of each row and then filling these boxes by $1, \ldots, (\rows{B}^t)_1$ from top to bottom.  Repeat this process with the leftmost unfilled box of each row, and continue until the $\ell-1$ leftmost boxes of each row have been filled, or a row has been completely filled if it has less than $\ell-1$ boxes.  Suppose a total of $m$ boxes has been filled to this point.  The shape that remains unfilled is exactly $\trim{\ell-1}{B}$.  For an example, 
see Figure~\ref{fig:main_proof}. 

\begin{figure}
\[
\begin{tikzpicture}[scale=0.45]
\begin{scope}
\fill[pink] (2,0) rectangle (3,1);
\fill[pink] (3,3) rectangle (5,4);
\fill[pink] (6,4) rectangle (7,5);
\draw[thick] (0,0) -- (0,1) -- (1,1) -- (1,4) -- (4,4) -- (4,5) --  (6,5) -- (6,6) -- (7,6) -- (7,4) -- (5,4) -- (5,3) -- (3,3) -- (3,0) -- cycle;
\draw (1,0) -- (1,1) -- (3,1);
\draw (2,0) -- (2,4);
\draw (3,3) -- (3,4);
\draw (4,3) -- (4,4) -- (5,4);
\draw (5,4) -- (5,5) -- (7,5);
\draw (6,4) -- (6,5);
\draw (1,3) -- (2,3);
\draw (2,3) -- (2,4);
\draw (3,3) -- (3,4);
\draw (1,2) -- (3,2);
\draw (1,3) -- (3,3);
\draw(0.5,0.5) node {6};
\draw(1.5,0.5) node {11};
\draw(2.5,0.5) node {12};
\draw(1.5,1.5) node {5};
\draw(2.5,1.5) node {10};
\draw(1.5,2.5) node {4};
\draw(2.5,2.5) node {9};
\draw(1.5,3.5) node {3};
\draw(2.5,3.5) node {8};
\draw(3.5,3.5) node {13};
\draw(4.5,3.5) node {14};
\draw(4.5,4.5) node {2};
\draw(5.5,4.5) node {7};
\draw(6.5,4.5) node {15};
\draw(6.5,5.5) node {1};
\draw(-1.5,3) node {$B=$};
\end{scope}
\begin{scope}[xshift=100mm]
\fill[pink] (2,1) rectangle (3,2);
\fill[pink] (5,3) rectangle (6,4);
\fill[pink] (6,6) rectangle (7,7);
\draw[thick] (0,0) -- (0,2) -- (2,2) -- (2,3) -- (3,3) -- (3,4) -- (4,4) -- (4,7) -- (7,7) -- (7,6) -- (6,6) -- (6,3) -- (3,3) -- (3,1) -- (1,1) -- (1,0) -- cycle;
\draw (0,1) -- (1,1) -- (1,2) -- (3,2);
\draw (2,1) -- (2,2);
\draw (4,3) -- (4,4) -- (6,4);
\draw (4,5) -- (6,5);
\draw (4,6) -- (6,6) -- (6,7);
\draw (5,3) -- (5,7);
\draw(0.5,0.5) node {11};
\draw(0.5,1.5) node {6};
\draw(1.5,1.5) node {10};
\draw(2.5,1.5) node {12};
\draw(2.5,2.5) node {5};
\draw(3.5,3.5) node {4};
\draw(4.5,3.5) node {13};
\draw(5.5,3.5) node {14};
\draw(4.5,4.5) node {3};
\draw(5.5,4.5) node {9};
\draw(4.5,5.5) node {2};
\draw(5.5,5.5) node {8};
\draw(4.5,6.5) node {1};
\draw(5.5,6.5) node {7};
\draw(6.5,6.5) node {15};
\draw(-1.5,3) node {$A=$};
\end{scope}
\begin{scope}[xshift=180mm]
\fill[pink] (2,1) rectangle (3,2);
\fill[pink] (5,3) rectangle (6,4);
\fill[pink] (6,6) rectangle (7,7);
\draw[thick] (2,1) rectangle (3,2);
\draw[thick] (4,3) rectangle (6,4);
\draw[thick] (6,6) rectangle (7,7);
\draw (5,3) -- (5,4);
\draw(2.5,1.5) node {1};
\draw(4.5,3.5) node {2};
\draw(5.5,3.5) node {3};
\draw(6.5,6.5) node {4};
\draw(0.5,3) node {$C=$};
\end{scope}
\end{tikzpicture}
\]
\caption{An example of the fillings of $B$, $A$ and $C$ from the proof of Theorem~\ref{thm:fsupport}.  Here, $\ell=3$, $m=11$, and the boxes of $\trim{2}{B}$ and $\trim{2}{A}$ are colored/shaded.}
\label{fig:main_proof}
\end{figure}
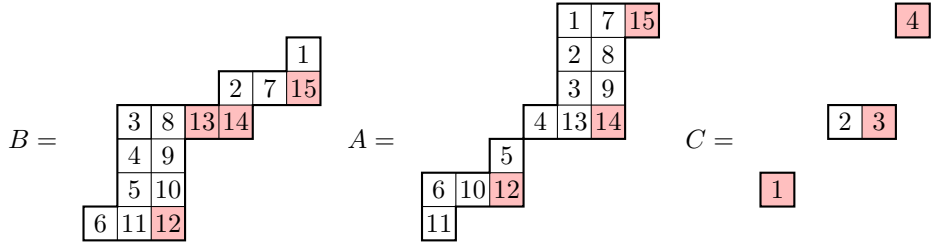
We now fill this remaining shape $\trim{\ell-1}{B}$ in the most dominant way possible.  Following Proposition~\ref{pro:f_extreme_fillings}(a), the descent composition of this remaining filling will be $\cols{\trim{\ell-1}{B}}^t$.  By Lemma~\ref{lem:trim}(b), this equals $\colsk{\ell}{B}^t$.  This might suggest, at first glance, that the descent composition $\alpha$ of $T$ consists of the concatenation of some composition of $m$ with $\colsk{\ell}{B}^t$.  This is not the case since $m$ is not a descent in $T$, but this will not affect our argument. 
  
Since $\fsupp{A} \supseteq \fsupp{B}$, there exists an SYT $T'$ of shape $A$ with descent composition $\alpha$.  Remove the boxes filled with $1, 2, \ldots, m$ in $T'$ to get a filling of some shape $C$, and subtract $m$ from all the entries of $C$.  This yields an SYT of shape $C$ with descent composition $\colsk{\ell}{B}^t$.  By Proposition~\ref{pro:f_extreme_fillings}(a)  and since $\colsk{\ell}{B}^t$ is weakly decreasing, we have $\colsk{\ell}{B}^t \domleq \cols{C}^t$.  Since $\colsk{\ell}{B}^t$ and $\cols{C}^t$ are both partitions of $|B|-m$, we deduce that $\colsk{\ell}{B} \domgeq \cols{C}$.

Now consider $\trim{\ell-1}{A}$.  Since $T'$ has descent composition $\alpha$, the numbers $1, 2, \ldots, m$ must have formed $\ell-1$ vertical strips that filled the left ends of any rows they occupied.  Therefore, $\trim{\ell-1}{A} \subseteq C$, by definition of $C$.  By Lemma~\ref{lem:dom_containment} applied to column lengths, $\cols{\trim{\ell-1}{A}} \domleq \cols{C}$.  Putting everything together, we get
\[
\cols{\trim{\ell-1}{A}} \domleq \cols{C} \domleq \colsk{\ell}{B}.
\]
Applying Lemma~\ref{lem:trim}(b) yields the desired result.
\end{proof}

\section{Conjecture for the converse}\label{sec:mainconjecture}

In Theorem~\ref{thm:fsupport} and its corollaries, our hypotheses on $A$ and $B$ have implied that we only consider cases where $A$ and $B$ have equal size.   Along the same lines, when comparing $\rowsk{k}{A}$ and $\colsk{k}{B}$ in this section, we will restrict to the case of $A$ and $B$ having the same size, and we can do so without our work losing any substance.

\subsection{The converse statements}

The converse of Corollary~\ref{cor:fpositive} would state that if $\rowsk{k}{A} \domleq \rowsk{k}{B}$ for all $k$ then $s_A - s_B$ is $F$-positive, but this is certainly not true.  To obtain a counterexample, one only needs to consider skew shapes of size 4: take $A = 311/1$ and $B = 32/1$; there are two SYT of shape $B$ with descent composition 22, but only one such SYT of shape A.  The same example shows that both possibilities for the converse of Theorem~\ref{thm:combine} also fail to hold.  As for the equality questions, $A_1$ and $A_2$ from Example~\ref{exa:intro} show that $\rowsk{k}{A} = \rowsk{k}{B}$ for all $k$ does not imply that $s_A = s_B$ or even that Schur supports are equal.  

Given these counterexamples, one might expect the converse of Theorem~\ref{thm:fsupport} to fail for a similarly low value of $|A|$, such as 4, 5 or 6.  However, we have computationally checked that the following conjecture holds for all $A$ and $B$ with $|A| \leq 12$.  

\begin{conjecture}\label{con:fsupport}
Skew shapes $A$ and $B$ of the same size satisfy $\fsupp{A} \supseteq \fsupp{B}$ if and only if $\rowsk{k}{A} \domleq \rowsk{k}{B}$ for all $k$.
\end{conjecture}

By Proposition~\ref{pro:rectsineq}, we could equivalently replace the $\mathrm{rows}$ condition by the appropriate $\mathrm{cols}$ or $\mathrm{rects}$ condition.  A proof of Conjecture~\ref{con:fsupport} would also imply that $\rowsk{k}{A} = \rowsk{k}{B}$ for all $k$ if and only if $\fsupp{A} = \fsupp{B}$, and perhaps this latter statement would be an easier one to prove or disprove.  

Obviously, the ``only if'' direction of Conjecture~\ref{con:fsupport} is Theorem~\ref{thm:fsupport}. Despite evidence in favor of the ``if'' direction, this author still remains somewhat skeptical for the following reason.  Close examination of the proof of Theorem~\ref{thm:fsupport} suggests that $\fsupp{B}$ encodes more information than $\rowsk{k}{B}$ for all $k$ or equivalently $\colsk{\ell}{B}$ for all $\ell$, since only certain elements of the support were used in the proof of Theorem~\ref{thm:fsupport}.  Roughly speaking, we focussed on those compositions $\alpha$ in the support that were obtained by starting our filling in the least dominant way possible, and then filling the remainder $\trim{\ell-1}{B}$ in the most dominant way possible; for each $\ell$, we only used one element of $\fsupp{B}$ to isolate $\colsk{\ell}{B}$.

It can be helpful to view Conjecture~\ref{con:fsupport} in terms of two partially ordered sets.  For the first poset $\suppfn{n}$, the elements will be equivalence classes of skew shapes of size $n$, where the equivalence relation is $A \sim B$ if $\fsupp{A} = \fsupp{B}$;  the order relation will be $[A] \geq_\suppfn{n} [B]$ if $\fsupp{A} \supseteq \fsupp{B}$, where $[A]$ denotes the equivalence class of $A$.  For the second poset, $\ncn{n}$, the elements will be equivalence classes of skew shapes of size $n$, where the equivalence relation is $A \sim B$ if $\rowsk{k}{A} = \rowsk{k}{B}$ for all $k$.  The order relation for the second poset is  $[A] \geq_\ncn{n} [B]$ if $\rowsk{k}{A} \domleq \rowsk{k}{B}$ for all $k$, where $[A]$ now denotes the equivalence class of $A$ under this second equivalence relation.  Is is straightforward to check that Conjecture~\ref{con:fsupport} is equivalent to the statement that the posets $\suppfn{n}$ and $\ncn{n}$ are isomorphic under the map that sends the equivalence class $[A]$ in $\suppfn{n}$ to the equivalence class $[A]$ in $\ncn{n}$.  The poset for the case $n=6$ is shown in Figure~\ref{fig:fsupp_6}.  
\begin{figure}
\[
\begin{tikzpicture}[scale=1.05]
\tikzstyle{every node}=[draw, thick, shape=circle, inner sep=1pt];
\Yboxdim4pt;
\draw (3,0) node (33) {\young(\ \ \ ,\ \ \ )};
\draw (1,0) node (42) {\young(\ \ \ \ ,\ \ )};
\draw (7,0) node (222) {\young(\ \ ,\ \ ,\ \ )};
\draw (5,0) node (321) {\young(\ \ \ ,\ \ ,\ )};
\draw (9,0) node (2211) {\young(\ \ ,\ \ ,\ ,\ )};
\draw (-0.5,1.5) node (6) {\young(\ \ \ \ \ \ )};
\draw (2,1.5) node (43/1) {\young(:\ \ \ ,\ \ \ )};
\draw (0.5,1.5) node (51) {\young(\ \ \ \ \ ,\ )};
\draw (4,1.5) node (331/1) {\young(:\ \ ,\ \ \ ,\ )};
\draw (6,1.5) node (332/11) {\young(:\ \ ,\ \ ,\ \ )};
\draw (3,1.5) node (411) {\young(\ \ \ \ ,\ ,\ )};
\draw (8,1.5) node (2221/1) {\young(:\ ,\ \ ,\ \ ,\ )};
\draw (7,1.5) node (3111) {\young(\ \ \ ,\ ,\ ,\ )};
\draw (9.5,1.5) node (21111) {\young(\ \ ,\ ,\ ,\ ,\ )};
\draw (10.5,1.5) node (111111) {\young(\ ,\ ,\ ,\ ,\ ,\ )};
\draw (1,3) node (52/1) {\young(:\ \ \ \ ,\ \ )};
\draw (0,3) node (61/1) {\young(:\ \ \ \ \ ,\ )};
\draw (4,3) node (422/11) {\young(:\ \ \ ,:\ ,\ \ )};
\draw (3,3) node (431/11) {\young(:\ \ \ ,:\ \ ,\ )};
\draw (2,3) node (441/3) {\young(:::\ ,\ \ \ \ ,\ )};
\draw (8,3) node (3222/111) {\young(:\ \ ,:\ ,:\ ,\ \ )};
\draw (6,3) node (3311/2) {\young(::\ ,\ \ \ ,\ ,\ )};
\draw (7,3) node (3321/111) {\young(:\ \ ,:\ \ ,:\ ,\ )};
\draw (9,3) node (22111/1) {\young(:\ ,\ \ ,\ ,\ ,\ )};
\draw (10,3) node (222221/11111) {\young(:\ ,:\ ,:\ ,:\ ,:\ ,\ )};
\draw (1,4.5) node (53/2) {\young(::\ \ \ ,\ \ \ )};
\draw (0,4.5) node (62/2) {\young(::\ \ \ \ ,\ \ )};
\draw (3,4.5) node (421/1) {\young(:\ \ \ ,\ \ ,\ )};
\draw (4,4.5) node (442/22) {\young(::\ \ ,::\ \ ,\ \ )};
\draw (2,4.5) node (521/11) {\young(:\ \ \ \ ,:\ ,\ )};
\draw (7,4.5) node (3211/1) {\young(:\ \ ,\ \ ,\ ,\ )};
\draw (6,4.5) node (3311/11) {\young(:\ \ ,:\ \ ,\ ,\ )};
\draw (5,4.5) node (4221/111) {\young(:\ \ \ ,:\ ,:\ ,\ )};
\draw (9,4.5) node (22211/11) {\young(:\ ,:\ ,\ \ ,\ ,\ )};
\draw (8,4.5) node (32221/1111) {\young(:\ \ ,:\ ,:\ ,:\ ,\ )};
\draw (10,4.5) node (222211/1111) {\young(:\ ,:\ ,:\ ,:\ ,\ ,\ )};
\draw (0,6) node (63/3) {\young(:::\ \ \ ,\ \ \ )};
\draw (3,6) node (432/21) {\young(::\ \ ,:\ \ ,\ \ )};
\draw (2,6) node (531/21) {\young(::\ \ \ ,:\ \ ,\ )};
\draw (1,6) node (621/21) {\young(::\ \ \ \ ,:\ ,\ )};
\draw (7,6) node (3321/21) {\young(::\ ,:\ \ ,\ \ ,\ )};
\draw (6,6) node (4211/2) {\young(::\ \ ,\ \ ,\ ,\ )};
\draw (4,6) node (4211/11) {\young(:\ \ \ ,:\ ,\ ,\ )};
\draw (5,6) node (4421/221) {\young(::\ \ ,::\ \ ,:\ ,\ )};
\draw (8,6) node (32211/111) {\young(:\ \ ,:\ ,:\ ,\ ,\ )};
\draw (10,6) node (222111/111) {\young(:\ ,:\ ,:\ ,\ ,\ ,\ )};
\draw (9,6) node (333321/22221) {\young(::\ ,::\ ,::\ ,::\ ,:\ ,\ )};
\draw (2.33,7.5) node (542/32) {\young(:::\ \ ,::\ \ ,\ \ )};
\draw (1,7.5) node (631/31) {\young(:::\ \ \ ,:\ \ ,\ )};
\draw (5,7.5) node (4321/211) {\young(::\ \ ,:\ \ ,:\ ,\ )};
\draw (3.67,7.5) node (5321/221) {\young(::\ \ \ ,::\ ,:\ ,\ )};
\draw (7.67,7.5) node (33211/211) {\young(::\ ,:\ \ ,:\ ,\ ,\ )};
\draw (6.33,7.5) node (43321/2221) {\young(::\ \ ,::\ ,::\ ,:\ ,\ )};
\draw (9,7.5) node (333221/22211) {\young(::\ ,::\ ,::\ ,:\ ,:\ ,\ )};
\draw (2,9) node (642/42) {\young(::::\ \ ,::\ \ ,\ \ )};
\draw (4.4,9) node (5421/321) {\young(:::\ \ ,::\ \ ,:\ ,\ )};
\draw (3.2,9) node (6321/321) {\young(:::\ \ \ ,::\ ,:\ ,\ )};
\draw (5.6,9) node (43221/2211) {\young(::\ \ ,::\ ,:\ ,:\ ,\ )};
\draw (8,9) node (332211/2211) {\young(::\ ,::\ ,:\ ,:\ ,\ ,\ )};
\draw (6.8,9) node (444321/33321) {\young(:::\ ,:::\ ,:::\ ,::\ ,:\ ,\ )};
\draw (3,10.5) node (6421/421) {\young(::::\ \ ,::\ \ ,:\ ,\ )};
\draw (5,10.5) node (54321/3321) {\young(:::\ \ ,:::\ ,::\ ,:\ ,\ )};
\draw (7,10.5) node (443321/33221) {\young(:::\ ,:::\ ,::\ ,::\ ,:\ ,\ )};
\draw (4,12) node (64321/4321) {\young(::::\ \ ,:::\ ,::\ ,:\ ,\ )};
\draw (6,12) node (554321/44321) {\young(::::\ ,::::\ ,:::\ ,::\ ,:\ ,\ )};
\draw (5,13.5) node (654321/54321) {\young(:::::\ ,::::\ ,:::\ ,::\ ,:\ ,\ )};
\draw [color=blue,thick] (6) -- (61/1);
\draw [color=blue,thick] (33) -- (43/1);
\draw [color=blue,thick] (33) -- (331/1);
\draw [color=blue,thick] (42) -- (43/1);
\draw [color=blue,thick] (42) -- (52/1);
\draw [color=blue,thick] (42) -- (441/3);
\draw [color=blue,thick] (43/1) -- (53/2);
\draw [color=blue,thick] (43/1) -- (431/11);
\draw [color=blue,thick] (51) -- (52/1);
\draw [color=blue,thick] (51) -- (61/1);
\draw [color=blue,thick] (52/1) -- (53/2);
\draw [color=blue,thick] (52/1) -- (62/2);
\draw [color=blue,thick] (52/1) -- (521/11);
\draw [color=blue,thick] (53/2) -- (63/3);
\draw [color=blue,thick] (53/2) -- (531/21);
\draw [color=blue,thick] (61/1) -- (62/2);
\draw [color=blue,thick] (62/2) -- (63/3);
\draw [color=blue,thick] (62/2) -- (621/21);
\draw [color=blue,thick] (63/3) -- (631/31);
\draw [color=blue,thick] (222) -- (332/11);
\draw [color=blue,thick] (222) -- (2221/1);
\draw [color=blue,thick] (321) -- (331/1);
\draw [color=blue,thick] (321) -- (332/11);
\draw [color=blue,thick] (321) -- (422/11);
\draw [color=blue,thick] (321) -- (3311/2);
\draw [color=blue,thick] (331/1) -- (431/11);
\draw [color=blue,thick] (331/1) -- (3311/11);
\draw [color=blue,thick] (332/11) -- (442/22);
\draw [color=blue,thick] (332/11) -- (3321/111);
\draw [color=blue,thick] (411) -- (422/11);
\draw [color=blue,thick] (411) -- (441/3);
\draw [color=blue,thick] (421/1) -- (432/21);
\draw [color=blue,thick] (421/1) -- (531/21);
\draw [color=blue,thick] (421/1) -- (4211/11);
\draw [color=blue,thick] (422/11) -- (421/1);
\draw [color=blue,thick] (422/11) -- (4221/111);
\draw [color=blue,thick] (431/11) -- (421/1);
\draw [color=blue,thick] (431/11) -- (442/22);
\draw [color=blue,thick] (432/21) -- (542/32);
\draw [color=blue,thick] (432/21) -- (4321/211);
\draw [color=blue,thick] (441/3) -- (421/1);
\draw [color=blue,thick] (441/3) -- (521/11);
\draw [color=blue,thick] (442/22) -- (432/21);
\draw [color=blue,thick] (442/22) -- (4421/221);
\draw [color=blue,thick] (521/11) -- (531/21);
\draw [color=blue,thick] (521/11) -- (621/21);
\draw [color=blue,thick] (531/21) -- (542/32);
\draw [color=blue,thick] (531/21) -- (631/31);
\draw [color=blue,thick] (531/21) -- (5321/221);
\draw [color=blue,thick] (542/32) -- (642/42);
\draw [color=blue,thick] (542/32) -- (5421/321);
\draw [color=blue,thick] (621/21) -- (631/31);
\draw [color=blue,thick] (631/31) -- (642/42);
\draw [color=blue,thick] (631/31) -- (6321/321);
\draw [color=blue,thick] (642/42) -- (6421/421);
\draw [color=blue,thick] (2211) -- (2221/1);
\draw [color=blue,thick] (2211) -- (3222/111);
\draw [color=blue,thick] (2211) -- (22111/1);
\draw [color=blue,thick] (2221/1) -- (3321/111);
\draw [color=blue,thick] (2221/1) -- (22211/11);
\draw [color=blue,thick] (3111) -- (3222/111);
\draw [color=blue,thick] (3111) -- (3311/2);
\draw [color=blue,thick] (3211/1) -- (3321/21);
\draw [color=blue,thick] (3211/1) -- (4211/2);
\draw [color=blue,thick] (3211/1) -- (32211/111);
\draw [color=blue,thick] (3222/111) -- (3211/1);
\draw [color=blue,thick] (3222/111) -- (32221/1111);
\draw [color=blue,thick] (3311/2) -- (3211/1);
\draw [color=blue,thick] (3311/2) -- (4221/111);
\draw [color=blue,thick] (3311/11) -- (3321/21);
\draw [color=blue,thick] (3311/11) -- (4421/221);
\draw [color=blue,thick] (3321/21) -- (4321/211);
\draw [color=blue,thick] (3321/21) -- (33211/211);
\draw [color=blue,thick] (3321/111) -- (3211/1);
\draw [color=blue,thick] (3321/111) -- (3311/11);
\draw [color=blue,thick] (4211/2) -- (4321/211);
\draw [color=blue,thick] (4211/2) -- (43321/2221);
\draw [color=blue,thick] (4211/11) -- (4321/211);
\draw [color=blue,thick] (4211/11) -- (5321/221);
\draw [color=blue,thick] (4221/111) -- (4211/2);
\draw [color=blue,thick] (4221/111) -- (4211/11);
\draw [color=blue,thick] (4321/211) -- (5421/321);
\draw [color=blue,thick] (4321/211) -- (43221/2211);
\draw [color=blue,thick] (4421/221) -- (4321/211);
\draw [color=blue,thick] (5321/221) -- (5421/321);
\draw [color=blue,thick] (5321/221) -- (6321/321);
\draw [color=blue,thick] (5421/321) -- (6421/421);
\draw [color=blue,thick] (5421/321) -- (54321/3321);
\draw [color=blue,thick] (6321/321) -- (6421/421);
\draw [color=blue,thick] (6421/421) -- (64321/4321);
\draw [color=blue,thick] (21111) -- (22111/1);
\draw [color=blue,thick] (21111) -- (222221/11111);
\draw [color=blue,thick] (22111/1) -- (22211/11);
\draw [color=blue,thick] (22111/1) -- (32221/1111);
\draw [color=blue,thick] (22111/1) -- (222211/1111);
\draw [color=blue,thick] (22211/11) -- (32211/111);
\draw [color=blue,thick] (22211/11) -- (222111/111);
\draw [color=blue,thick] (32211/111) -- (33211/211);
\draw [color=blue,thick] (32211/111) -- (43321/2221);
\draw [color=blue,thick] (32211/111) -- (333221/22211);
\draw [color=blue,thick] (32221/1111) -- (32211/111);
\draw [color=blue,thick] (32221/1111) -- (333321/22221);
\draw [color=blue,thick] (33211/211) -- (43221/2211);
\draw [color=blue,thick] (33211/211) -- (332211/2211);
\draw [color=blue,thick] (43221/2211) -- (54321/3321);
\draw [color=blue,thick] (43221/2211) -- (443321/33221);
\draw [color=blue,thick] (43321/2221) -- (43221/2211);
\draw [color=blue,thick] (43321/2221) -- (444321/33321);
\draw [color=blue,thick] (54321/3321) -- (64321/4321);
\draw [color=blue,thick] (54321/3321) -- (554321/44321);
\draw [color=blue,thick] (64321/4321) -- (654321/54321);
\draw [color=blue,thick] (111111) -- (222221/11111);
\draw [color=blue,thick] (222111/111) -- (333221/22211);
\draw [color=blue,thick] (222211/1111) -- (222111/111);
\draw [color=blue,thick] (222211/1111) -- (333321/22221);
\draw [color=blue,thick] (222221/11111) -- (222211/1111);
\draw [color=blue,thick] (332211/2211) -- (443321/33221);
\draw [color=blue,thick] (333221/22211) -- (332211/2211);
\draw [color=blue,thick] (333221/22211) -- (444321/33321);
\draw [color=blue,thick] (333321/22221) -- (333221/22211);
\draw [color=blue,thick] (443321/33221) -- (554321/44321);
\draw [color=blue,thick] (444321/33321) -- (443321/33221);
\draw [color=blue,thick] (554321/44321) -- (654321/54321);
\end{tikzpicture}
\]
\caption{$\suppfn{6} = \ncn{6}$.  One representative of each equivalence class is drawn.}
\label{fig:fsupp_6}
\end{figure}
%

\subsection{Special cases of the conjecture}

It is simple to show that Conjecture~\ref{con:fsupport} holds for horizontal strips.  Indeed, 
$s_A$ for a horizontal strip $A$ is completely determined by $\rowsk{1}{A}=\rows{A}$.  In fact, we see that $s_A$ in this case is the complete homogeneous symmetric function $h_{\rows{A}}$.  It is well known (see, for example, \cite[Example~I.7.9(b)]{Mac95}) that $h_{\rows{A}}-h_{\rows{B}}$ is Schur-positive if and only if $\rows{A} \domleq \rows{B}$. Thus, if $\rows{A} \domleq \rows{B}$, then $s_A - s_B$ is Schur-positive, which implies that $\fsupp{A} \supseteq \fsupp{B}$.  

In Section~\ref{sec:multfree}, we will completely determine the poset $\suppfn{n}$ restricted to $F$-multiplicity free skew shapes (in which case $\fsupp{A} \supseteq \fsupp{B}$ is equivalent to $s_A - s_B$ being $F$-positive), from which it will follow that Conjecture~\ref{con:fsupport} holds in that case.

The remainder of this section is devoted to a proof of Conjecture~\ref{con:fsupport} for a special class of ribbons, which we now define.

\begin{definition}
A ribbon is said to be \emph{elongated} if all its rows have length at least two.
\end{definition}

\begin{theorem}\label{thm:ribbons}
Elongated ribbons $A$ and $B$ of the same size satisfy
$\fsupp{A} \supseteq \fsupp{B}$ if and only if $\rowsk{k}{A} \domleq \rowsk{k}{B}$ for all $k$.
\end{theorem}

\begin{proof}
By Theorem~\ref{thm:fsupport}, we need only prove the ``if'' direction.  A key simplification for elongated ribbons is that $\rowsk{k}{A} \domleq \rowsk{k}{B}$ for all $k$ is equivalent to $\rowsk{k}{A} \domleq \rowsk{k}{B}$ for $k=1,2$.   

So first suppose $\rowsk{1}{A} = \rows{A} \domleq \rows{B}$ for elongated ribbons $A$ and $B$.  This implies that $A$ has at least as many (nonempty) rows as $B$.  On the other hand, $\rowsk{2}{A}$ is just a sequence of ones of length equal to one less than the number of rows of $A$.  Thus $\rowsk{2}{A} \domleq \rowsk{2}{B}$ implies that $B$ has at least as many rows as $A$.  So our rows condition is equivalent to the fact that $\rows{A} \domleq \rows{B}$ and that $A$ and $B$ have an equal number of rows.  

Our proof is facilitated by \cite[Theorem~3.3]{KWvW08}, which considers ribbons whose row lengths from top to bottom are weakly decreasing.  In this case, their theorem says that $s_A - s_B$ is Schur-positive if and only if $\rows{A} \domleq \rows{B}$ and $A$ and $B$ have equal numbers of rows.  For our purposes, we get that if $A$ and $B$ are elongated ribbons with weakly decreasing rows lengths and $\rowsk{k}{A} \domleq \rowsk{k}{B}$ for all $k$, then $\fsupp{A} \supseteq \fsupp{B}$.  Therefore, it suffices to show that for elongated ribbons $A$, we have $
\fsupp{A} = \fsupp{A^\geq}$, where $A^\geq$ denotes the ribbon obtained from $A$ by sorting its row lengths into weakly decreasing order from top to bottom.  Moreover, it suffices to show that the $F$-support of an elongated ribbon $A$ is preserved when we switch two adjacent rows of $A$ where the lower row is longer than the upper row; this is the result for which we will now give a combinatorial proof.

Consider the setup shown in Figure~\ref{fig:ribbon_proof}.  This shows two adjacent rows of an SYT $T$ with descent composition $\alpha$ and shape $A$, where $A$ is an elongated ribbon.  We assume that row $i+1$ is strictly longer than row $i$, i.e, that $\ell > k$. Starting with $T$, we wish to form an SYT with descent composition $\alpha$ and shape $A'$, where $A'$ is obtained from $A$ by switching the lengths of the rows $i$ and $i+1$.  Our plan is to move $\ell-k$ entries of $T$ from row $i+1$ up to row $i$ so as to preserve the descent set.  After moving entries, we will sort our rows into weakly decreasing order.  We will need to check that the result still has descent composition $\alpha$ and that the columns are strictly increasing at each of the places marked with the thick lines in columns $j_1$, $j_2$ and $j_3$.
\begin{figure}
\[
\begin{tikzpicture}[scale=0.45]
\draw (0,0) rectangle (1,1);
\draw (0,1) rectangle (1,2);
\draw (1,1) rectangle (5,2);
\draw (5,1) rectangle (6,2);
\draw (5,2) rectangle (6,3);
\draw (6,2) rectangle (8,3);
\draw (8,2) rectangle (9,3);
\draw (8,3) rectangle (9,4);
\draw[dashed] (0,0) -- (-1.8,0);
\draw[dashed] (0,1) -- (-1.8,1);
\draw[dashed] (9,3) -- (10.8,3);
\draw[dashed] (9,4) -- (10.8,4);
\draw (0.5,0.5) node {$d$};
\draw (0.5,1.5) node {$c_1$};
\draw (5.5,1.5) node {$c_\ell$};
\draw (5.5,2.5) node {$b_1$};
\draw (8.5,2.5) node {$b_k$};
\draw (8.5,3.5) node {$a$};
\draw (-6,1.5) node[anchor=west] {row $i+1$};
\draw (-6,2.5) node[anchor=west] {row $i$};
\fill (0,0.9) rectangle (1,1.1);
\fill (5,1.9) rectangle (6,2.1);
\fill (8,2.9) rectangle (9,3.1);
\draw (0.5,-1) node {$j_1$};
\draw (5.5,-1) node {$j_2$};
\draw (8.5,-1) node {$j_3$};
\draw (-2,-1) node {columns};
\end{tikzpicture}
\]
\caption{The setup for the proof of Theorem~\ref{thm:ribbons}.}
\label{fig:ribbon_proof}
\end{figure}
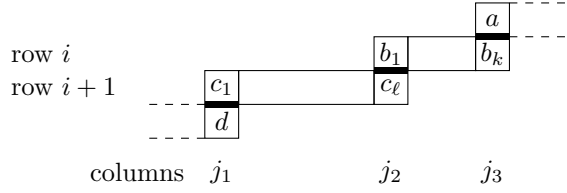

To begin, if $b_r$ in row $i$ is a descent and $b_r+1$ appears in row $i+1$ as $c_s$, then we consider $b_r$ and $c_s$ as being \emph{paired}.  Paired elements will always remain in their current rows except as described below.  There are two cases to consider according to whether or not there exist paired elements.  

Suppose that there is at least one pair.  We know there are at least $\ell-k$ non-paired entries in row $i+1$, so move the largest $\ell-k$ non-paired entries of row $i+1$ up to row $i$.  Since paired elements remain in their current rows, the descent set is preserved.  Since row $i$ only gains elements, there will still be a strict increase in column $j_3$.  Since there is at least one pair, there will still be a strict increase in column $j_2$.   In most cases, we will keep a strict increase in column $j_1$ since we moved the largest non-paired elements out of row $i+1$.  
However, consider the remaining case when there is a full set of $k$ pairs and $c_1$ gets moved, resulting in the loss of the strict increase in column $j_1$.  Suppose that, after this moving takes place, $c_s$ is the entry in the leftmost box of row $i+1$.  Since $c_s > d$, our technique will be to switch $c_s$ and $d$.  
The result will clearly be an SYT.  We know that $c_s$ is paired, with $c_s-1$ appearing in row $i$.  Therefore, $c_s-1$ will remain a descent.  Since $c_s > d$, whether or not $c_s$, $d$ or $d-1$ are descents will be unaffected by the switch of $c_s$ and $d$.  We conclude that the descent set is preserved, and we have the desired SYT of shape $A'$ and descent set $\alpha$.  

Now suppose there are no pairs.  In this case, read along row $i+1$ from right to left.  Taking the elements $c_s$ of row $i+1$ one at a time, move $c_s$ up to the appropriate spot along row $i$.  However, if doing so would violate the strict increase in column $j_2$, then leave $c_s$ in row $i+1$ and move on to consider $c_{s-1}$, stopping once we have moved up $\ell-k$ elements.  If $c_s$ remains in the bottom row, all subsequent elements $c_t$ will be able to move up to row $i$, since $c_t < c_s$.  As before, there will still be a strict increase in column $j_3$.  By design, there will be a strict increase in column $j_2$.  Since $k\geq 2$, $c_1$ will remain in position, thus preserving the strict increase in column $j_1$.  Since there are no paired elements, the only way we could change the descent set would be if $c_s$ stayed in row $i+1$ while $c_s-1$ moved from row $i+1$ to row $i$.  This would imply that $c_s-1=c_{s-1}$, which is impossible for the following reason: when we attempted to move $c_s$ up to row $i$ and failed, it must have been because the entry $b$ in the leftmost box of row $i$ at that time was strictly between the values $c_{s-1}$ and $c_s$.  We conclude that the descent set is preserved, as required.  
\end{proof}

One might wonder if the row overlap condition might imply something stronger than $F$-support containment in the special case of elongated ribbons.  More\linebreak precisely, does Theorem~\ref{thm:ribbons} still hold if we replace the condition ``$\fsupp{A} \supseteq \fsupp{B}$'' by ``$\ssupp{A} \supseteq \ssupp{B}$'' or by ``$s_A - s_B$ is $F$-positive''?  The answer is ``no'' for both possibilities, as can be seen by letting $A=632/21$ and $B=652/41$.  

One obvious next step would be to try to prove Conjecture~\ref{con:fsupport} for general ribbons.  In that regard, we note that the method of proof above has some freedom that we did not use.  First observe that the moving of elements described above also works if we want to move less than $\ell-k$ elements.  Perhaps more importantly, we started with \emph{any} given $T$ of shape $A$ and descent composition $\alpha$.  However, since we are only proving a result about supports, it is sufficient to choose a ``special'' or particular $T$ of shape $A$ and descent set $\alpha$, and there might be a helpful way to make this choice.

\subsection{A saturation-type consequence of the conjecture}

If Conjecture~\ref{con:fsupport} were true, we would get a version of the Saturation Theorem for skew shapes, as we now explain.  

For a partition $\lambda$ and a positive integer $n$, let $n\lambda$ denote the partition obtained by multiplying all the parts of $\lambda$ by $n$.  The Saturation Theorem \cite{KnTa99} (see also \cite{Buc00} and the survey \cite{Ful00}) can be stated in the following way: for partitions $\lambda$, $\mu$, $\nu$ and any positive integer $n$, we have
\[
\ssupp{\lambda/\mu} \supseteq \ssupp{\nu} \mbox{\ \ if and only if\ \ } \ssupp{n\lambda/n\mu} \supseteq \ssupp{n\nu}.
\]
This statement is written here in an overly complicated form since $\ssupp{\nu}$ is obviously just $\{\nu\}$ and similarly for $n\nu$, but the statement is in the form we need for the following analogue.  For a skew shape $A = \lambda/\mu$, we define $nA = n\lambda/n\mu$.  Then, David Speyer asked the author if the following skew analogue of the Saturation Theorem could possibly be true: for skew shapes $A$ and $B$ and any positive integer $n$,
\[
\ssupp{A} \supseteq \ssupp{B} \mbox{\ \ if and only if\ \ } \ssupp{nA} \supseteq \ssupp{nB}.
\]
This is false in the ``only if'' direction (which is the easy direction for the Saturation Theorem) since 
\[
\ssupp{4311/21} \supseteq \ssupp{4421/311}
\]
but 633 is contained in
\[
\ssupp{8842/622} \setminus \ssupp{8622/42}.
\]
We do not know of a counterexample for the ``if'' direction.

Alejandro Morales asked about connections between the present paper and the Saturation Theorem, and there does appear to be hope of a skew analogue of the Saturation Theorem if we move to $F$-supports.

\begin{question}\label{que:saturation}
For skew shapes $A$ and $B$ and any positive integer $n$, is it the case that 
\[
\fsupp{A} \supseteq \fsupp{B} \mbox{\ \ if and only if\ \ } \fsupp{nA} \supseteq \fsupp{nB}\ ?
\]
\end{question}

Since dominance order is preserved under the map that sends $\lambda$ to $n\lambda$ and the inverse map, a proof of Conjecture \ref{con:fsupport} would imply an affirmative answer to Question~\ref{que:saturation}.

\section{$F$-multiplicity-free skew shapes}\label{sec:multfree}

In \cite[Theorem~3.4]{BevW13}, Bessenrodt and van Willigenburg give a complete classification of those skew shapes $A$ that are \emph{$F$-multiplicity-free}, meaning that when $s_A$ is expanded in the $F$-basis, all the coefficients are 0 or 1.  In other words, $A$ is $F$-multiplicity-free if and only if all SYTx of shape $A$ have distinct descent sets.  Our first goal for this section is to completely classify those $F$-multiplicity-free $A$ and $B$ such that $s_A-s_B$ is $F$-positive.  By the definition of $F$-multiplicity-free, this is equivalent to classifying those $A$ and $B$ such that $\fsupp{A} \supseteq \fsupp{B}$.  Our second goal is to show that this classification implies the truth of Conjecture~\ref{con:fsupport} in the case of $F$-multiplicity-free shapes.

Let us begin with the aforementioned result from \cite{BevW13}.  Let $A^\circ$ denote $A$ rotated $180^\circ$, also known as the \emph{antipodal rotation} of $A$.  As is well known \cite[Exercise~7.56(a)]{ec2}, $s_A = s_{A^\circ}$.  Let us use $1^\ell$ to denote a sequence of $\ell$ copies of $1$, and $A \oplus B$ to denote the skew shape obtained by positioning $A$ immediately below and to the left of $B$ in such a way that $A$ and $B$ have no rows or columns in common.  For example, $(1^2) \oplus (2)$ can also be written as $311/1$.

\begin{theorem}[\cite{BevW13}]\label{thm:bevw}  
A skew shape $A$ of size $n$
is $F$-multiplicity-free if and only if, up to transpose, $A$ or $A^\circ$ is one of 
\renewcommand{\theenumi}{\roman{enumi}}
\begin{enumerate}
\item $(3,3)$ if $n=6$,
\item $(4,4)$ if $n=8$,
\item $(n-2,2)$ if $n\geq4$,
\item $(n-\ell,1^\ell)$ for $0 \leq \ell \leq n-1$,
\item $(1^\ell) \oplus (n-\ell)$ for $1 \leq \ell \leq n-1$. 
\end{enumerate}
\end{theorem}

Notice that the first four types in the list above are \emph{straight shapes}, meaning that they take the form $\lambda/\emptyset$ for some partition $\lambda$.  We now state the main result of this section.

\begin{theorem}\label{thm:multfree}
Let $A$ and $B$ be $F$-multiplicity-free skew shapes of size $n$.  Then $s_A = s_B$ if and only if $B \in \{A, A^\circ\}$.  Otherwise, $s_A - s_B$ is $F$-positive (equivalently $\fsupp{A} \supseteq \fsupp{B}$) if and only if one of the following conditions holds up to antipodal rotation of $A$ and/or $B$:
\begin{enumerate}
\item $A = (1^\ell) \oplus (n-\ell)$ and $B \in \{(n-\ell, 1^\ell), (n-\ell+1, 1^{\ell-1})\}$ for $1 \leq \ell \leq n-1$;
\item $A = (1^2) \oplus (n-2)$ and $B = (n-2, 2)$ with $n\geq4$;
\item $A= (1^{n-2}) \oplus (2)$ and $B = (2,2, 1^{n-4})$ with $n\geq4$.
\end{enumerate}
\end{theorem}

Observe that the skew shapes $A$ and $B$ in (c) are just the transposes of those in (b), while the transposes of $A$ and $B$ from (a) will be another pair from (a).  The subposet of $\suppfn{5}$ consisting of $F$-multiplicity-free skew shapes is depicted in Figure~\ref{fig:suppfn4}.
\begin{figure}
\[
\begin{tikzpicture}[scale=1.4]
\tikzstyle{every node}=[draw, thick, shape=circle, inner sep=3pt]; 
\Yboxdim6pt;
\draw (0,0) node (5) {\yng(5)};
\draw (2,0) node (41) {\yng(4,1)};
\draw (3,0) node (32) {\yng(3,2)};
\draw (4,0) node (311) {\yng(3,1,1)};
\draw (5,0) node (221) {\yng(2,2,1)};
\draw (6,0) node (2111) {\yng(2,1,1,1)};
\draw (8,0) node (11111) {\yng(1,1,1,1,1)};
\draw (1,1.5) node (51/1) {\young(:\ \ \ \ ,\ )};
\draw (3,1.5) node (411/1) {\young(:\ \ \ ,\ ,\ )};
\draw (5,1.5) node (3111/1) {\young(:\ \ ,\ ,\ ,\ )};
\draw (7,1.5) node (21111/1) {\young(:\ ,\ ,\ ,\ ,\ )};
\draw[color=blue,thick] (5) -- (51/1) -- (41) -- (411/1) -- (311) -- (3111/1) -- (2111) -- (21111/1) -- (11111);
\draw[color=blue,thick]  (32) -- (411/1);
\draw[color=blue,thick] (221) -- (3111/1);
\end{tikzpicture}
\]
\caption{The subposet of $\suppfn{5}$ consisting of $F$-multiplicity-free skew shapes.  For each $A$ drawn, $A^\circ$ is also a member of the equivalence class.}
\label{fig:suppfn4}
\end{figure}
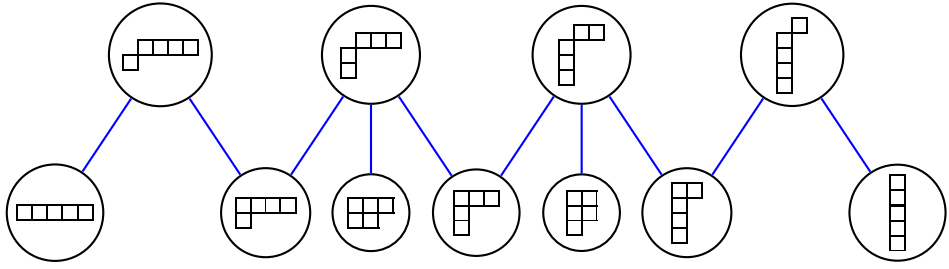

\begin{proof}[Proof of Theorem~\ref{thm:multfree}]  Since $s_A = s_{A^\circ}$, we know that if $B \in \{A, A^\circ\}$ then $s_A = s_B$.  The converse is a consequence of the analysis below that proves the bulk of the statement of the theorem. 

If $\fsupp{A} \supseteq \fsupp{B}$, then Proposition~\ref{pro:f_extreme_fillings} tells us that $\rows{A} \domleq \rows{B}$ and $\cols{A} \domleq \cols{B}$.  If $A$ and $B$ are straight shapes, then the latter inequality is equivalent to $\rows{A}^t \domleq \rows{B}^t$ and hence $\rows{A} \domgeq \rows{B}$.  Thus $\rows{A}=\rows{B}$ and so $A=B$.  Therefore, the straight shapes given in (i)--(iv) of Theorem~\ref{thm:bevw} are all incomparable according to $F$-support containment.

It remains to consider comparabilities involving the skew shapes $A=(1^\ell) \oplus (n-\ell)$ for $1 \leq \ell \leq n-1$.  Note that this class is mapped to itself under the transpose operation.  
The rest of the proof is a relatively routine checking of cases involving some explicit expansions of skew Schur functions.  Taking one of our skew shapes to be of type (v) of Theorem~\ref{thm:bevw}, we will work backwards through the five possibilities for the type of the other skew shape.  
\begin{itemize}
\item[(v)] Let us first consider the case when $A = (1^\ell) \oplus (n-\ell)$ and $B = (1^m) \oplus (n-m)$ for $\ell < m$.  We have $\rows{A} \succ \rows{B}$ but $\cols{A} \prec \cols{B}$.  Proposition~\ref{pro:f_extreme_fillings} then tells us that $\fsupp{A}$ and $\fsupp{B}$ are incomparable.

\item[(iv)] By the Pieri rule \cite[Theorem~7.15.7]{ec2},  
\begin{equation}\label{equ:multfree}
s_{(1^\ell) \oplus (n-\ell)} = s_{(n-\ell, 1^\ell)} + s_{(n-\ell+1, 1^{\ell-1})}.
\end{equation}
Therefore $s_A - s_B$ is Schur-positive and hence $F$-positive for $A$ and $B$ from (a) of the current theorem.  

We next consider other comparabilities among those skew shapes of types (iv) and (v) of Theorem~\ref{thm:bevw}.  Note that the class (iv) is also mapped to itself under the transpose operation.  From \cite[Lemma~3.2]{BevW13}, we know that for $n\geq 1$ and $0 \leq \ell \leq n-1$, we have
\begin{equation}\label{equ:hook}
s_{(n-\ell, 1^\ell)} = \sum_\alpha F_\alpha,
\end{equation}
where the sum if over all compositions $\alpha$ of size $n$ with $\ell+1$ parts.  Using this and~\eqref{equ:multfree}, we deduce that the only comparabilities that exist between skew shapes of types (iv) and (v)  are those already given in (a) of the current theorem. 

\item[(iii)] Consider $(n-2,2)$ for $n\geq4$.  Again we refer to \cite[Lemma~3.2]{BevW13} which gives
\[
s_{(n-2,2)} = \sum_{i=2}^{n-2} F_{(i,n-i)} + \sum_{j=3}^{n-1} \sum_{i=1}^{j-2} F_{(i, j-i, n-j)}.
\]
Comparing with~\eqref{equ:multfree} and~\eqref{equ:hook}, comparabilities of $(n-2,2)$ with skew shapes of type $(1^\ell) \oplus (n-\ell)$ can only occur when $\ell=2$.  In this case, we have
\[
s_{(1^2) \oplus (n-2)} = \sum_{i=1}^{n-1} F_{(i,n-i)} + \sum_{j=2}^{n-1} \sum_{i=1}^{j-1} F_{(i, j-i, n-j)},
\]
and so
\[
s_{(1^2) \oplus (n-2)}- s_{(n-2,2)} = F_{(1,n-1)} + F_{(n-1,1)} + \sum_{j=2}^{n-1} F_{(j-1,1,n-j)}.
\]
In particular, $s_A-s_B$ is $F$-positive for $A$ and $B$ from (b) of the current theorem.  

The usual $\omega$ involution \cite[\S7.6 and Theorem~7.15.6]{ec2} can be extended to quasisymmetric functions in a way that preserves $F$-positivity; see \cite[Exercise~7.94(a)]{ec2} for one such extension, and \cite{McWapr} for further details and references.  Since applying this extended $\omega$ preserves $F$-positivity, we draw the desired analogous conclusion for comparabilities involving $(n-2,2)^t = (2,2,1^{n-4})$ of (c).  

\item[(ii), (i)] Finally, by direct computation with $n=6$ and $n=8$, we get that there are no comparabilities involving $(3,3)$, $(4,4)$ or their transposes.  
\end{itemize}
\end{proof}

\begin{remark}
If $A$ is $F$-multiplicity-free then Theorem~\ref{thm:fexpansion} implies that $A$ is \emph{Schur-multiplicity-free}, defined in the natural way.  Thus when $A$ and $B$ are $F$-multiplicity-free,  $s_A - s_B$ being Schur-positive is equivalent to $\ssupp{A} \supseteq \ssupp{B}$.  It is relatively easy to determine exactly when $s_A - s_B$ is Schur-positive in the case that $A$ and $B$ are $F$-multiplicity-free, as we now describe.  The straight shapes from (i)--(iv) of Theorem~\ref{thm:bevw} are obviously incomparable. Then it follows from \eqref{equ:multfree} that the conditions for $s_A - s_B$ to be Schur-positive are exactly as in Theorem~\ref{thm:multfree} but with conditions (b) and (c) deleted.  

Determining conditions for $s_A - s_B$ to be Schur-positive when $A$ and $B$ are \emph{Schur}-multiplicity-free seems to be a significantly harder problem.  See \cite{McvW09b} for the case of ribbons and \cite{Gut09} for the solution to $s_A = s_B$ in the Schur-multiplicity-free situation.  Both of these papers rely on a classification of skew shapes that are Schur multiplicity-free, which was given in \cite{Gut10,ThYo10}.
\end{remark}

With Theorem~\ref{thm:multfree} in place, we can now give our last piece of evidence in favor of Conjecture~\ref{con:fsupport}.

\begin{corollary}
Conjecture~\ref{con:fsupport} holds when $A$ and $B$ are  $F$-multiplicity-free skew shapes.
\end{corollary}

\begin{proof}
We wish to find all pairs of $F$-multiplicity-free skew shapes $A$ and $B$ of the same size satisfying $\rowsk{k}{A} \domleq \rowsk{k}{B}$ for all $k$, and show that such $A$ and $B$ satisfy one of the conditions of Theorem~\ref{thm:multfree}.  Since antipodal rotation preserves $F$-supports and row overlaps, if Conjecture~\ref{con:fsupport} holds for $A$ and $B$, it will automatically hold with $A^\circ$ in place of $A$ and/or with $B^\circ$ in place of $B$.  Therefore, we only need to consider the five classes of $F$-multiplicity-free shapes listed in Theorem~\ref{thm:bevw} and not their antipodal rotations.

First suppose that $A$ and $B$ are straight shapes and that $\rowsk{k}{A} \domleq \rowsk{k}{B}$ for all $k$.  By Proposition~\ref{pro:rectsineq}, we then also know that $\colsk{k}{A} \domleq \colsk{k}{B}$ for all $k$.  We have $\rowsk{1}{A} = \colsk{1}{A}^t \domgeq \colsk{1}{B}^t = \rowsk{1}{B}$.  Thus $\rowsk{1}{A} = \rowsk{1}{B}$ and so $A=B$ up to antipodal rotation.  

It remains to consider the case when $A$ and/or $B$ takes the form $(1^\ell) \oplus (n-\ell)$ for $1 \leq \ell \leq n-1$.  If $A$ takes this form, then
\begin{equation}\label{equ:overlaps5}
(\rowsk{1}{A}, \ldots, \rowsk{\ell}{A}) = ((n-\ell,1^\ell), 1^{\ell-1}, 1^{\ell-2}, \ldots, 1)),
\end{equation}
Let us work in the reverse order through the five possibilities from Theorem~\ref{thm:bevw} for the type of $B$.
\begin{itemize}
\item[(v)] If $B = (1^m) \oplus (n-m)$ for $m\neq \ell$ then it will be neither true that $\rowsk{k}{A} \domleq \rowsk{k}{B}$ for all $k$, nor $\rowsk{k}{B} \domleq \rowsk{k}{A}$ for all $k$.  In this case we will say that the \emph{row overlap sequences are incomparable}, and there is nothing to prove.
\item[(vi)] If $B$ is of type (iv) from Theorem~\ref{thm:bevw}, then $B = (n-m, 1^m)$ for some $0 \leq m \leq n-1$ and we have
\begin{equation}\label{equ:overlaps4}
(\rowsk{1}{B}, \ldots, \rowsk{m+1}{B}) = ((n-m,1^m), 1^{m}, 1^{m-1}, \ldots, 1)),
\end{equation}
Comparing~\eqref{equ:overlaps5} and~\eqref{equ:overlaps4}, we see that the
relevant row overlap sequence comparabilities are that $\rowsk{k}{A} \domleq \rowsk{k}{B}$ for all $k$ when $m=\ell$ or $m=\ell-1$.  These two possibilities for $m$ give exactly the conditions of (a) of Theorem~\ref{thm:multfree}. 

\item[(iii)] Let $B = (n-2,2)$.   For $\ell>2$, we have $\rowsk{1}{A} \prec \rowsk{1}{B}$, but $\rowsk{3}{A} \succ \rowsk{3}{B} = \emptyset$, so the row overlap sequences of $A$ and $B$ are incomparable.  If $\ell=2$, we get that $\rowsk{k}{A} \domleq \rowsk{k}{B}$ for all $k$.  In this case, (b) of Theorem~\ref{thm:multfree} is satisfied.  If $\ell=1$, the row overlap sequences are again incomparable.  If $B=(n-2,2)^t=(2,2,1^{n-4})$, then the analogous conclusions can be drawn by using $\mathrm{cols}$ in place of $\mathrm{rows}$ and (c) of Theorem~\ref{thm:multfree}. 
\item[(ii), (i)] When $B$ equals $(3,3)$, $(4,4)$ or one of their transposes, it is routine to check that the row overlap sequences of $A$ and $B$ are incomparable.  
\end{itemize}
We conclude that in all cases where $\rowsk{k}{A} \domleq \rowsk{k}{B}$ for all $k$, one of the conditions of Theorem~\ref{thm:multfree} is satisfied, so we have $\fsupp{A} \supseteq \fsupp{B}$, as required.  
\end{proof}

\section{Other quasisymmetric bases}\label{sec:conclusion}

It is natural to ask if other quasisymmetric function bases have a role to play in comparing skew Schur functions.  In this section, we look at three other bases, namely 
\begin{itemize}
\item the monomial quasisymmetric functions of Equation~\eqref{equ:monomials};
\item the  \emph{quasisymmetric Schur} basis of Haglund et al.\  \cite{HLMvW11}, whose elements we denote by $\QS_\alpha$;
\item the \emph{dual immaculate} basis of Berg et al.~\cite{BBSSZpr}, whose elements we denote by $D_\alpha$.
\end{itemize}
Examples of expansions in these bases appear in Table~\ref{tab:bases}.  The latter two bases are both new (introduced in 2008 and 2012 respectively) and are the subject of considerable current interest.

\renewcommand{\arraystretch}{1.0} 
\setlength{\tabcolsep}{6pt} 
\begin{table}
\[
\begin{array}{llll}
& 
\begin{tikzpicture}[scale=0.4]
\begin{scope}
\draw (-1.5,1.5) node {$A_1=$};
\draw[thick] (0,0) -- (1,0) -- (1,2) -- (3,2) -- (3,3) -- (1,3) -- (1,2) -- (0,2) -- cycle;
\draw (2,3) -- (2,2);
\draw (0,1) -- (1,1);
\end{scope}
\end{tikzpicture}
&&
\begin{tikzpicture}[scale=0.4]
\begin{scope}
\draw (-1.5,1) node {$A_3=$};
\draw[thick] (0,0) -- (2,0) -- (2,2) -- (0,2) -- cycle;
\draw (1,0) -- (1,2);
\draw (0,1) -- (2,1);
\end{scope}
\end{tikzpicture}
\smallskip \\ 
\mbox{Schur expansion}
& 
s_{31} + s_{211}
&&
s_{22}
\smallskip \\ 
\mbox{$F$-expansion}
& 
F_{31} + F_{13} + F_{22} + {} 
&&
F_{22} + F_{121}
 \\
&
F_{211} + F_{121} + F_{112}
&&
\smallskip \\
\mbox{$M$-expansion}
& 
M_{31} + M_{13} + M_{22} + 3M_{211} + {}
&&
M_{22} + M_{211} + M_{121} + {}
\\
&
3M_{121} + 3M_{112} + 6M_{1111}
&&
M_{112} + 2M_{1111}
\smallskip \\
\mbox{$\QS$-expansion}
& 
\QS_{31} + \QS_{13} + \QS_{211} + \QS_{121} + \QS_{112}
&&
\QS_{22}
\smallskip \\
\mbox{$D$-expansion}
&
D_{31} + D_{211}
&&
D_{22} - D_{13}
\end{array}
\]
\caption{The expansion of two skew Schur functions from Example~\ref{exa:intro} in the bases of Section~\ref{sec:conclusion}.}
\label{tab:bases}
\end{table}

Our goals for this section are to show all the implications appearing in Figure~\ref{fig:moreimplications} that did not already appear in Figure~\ref{fig:implications}, to show that all the implications except the rightmost one are strict in the sense that the converse implications are false, and to show that the set of implications in Figure~\ref{fig:moreimplications} is complete in certain sense. 
\renewcommand{\arraystretch}{0} 
\setlength{\tabcolsep}{0pt} 
\begin{figure}
\begin{center}
\begin{tikzpicture}[scale=0.8]
\tikzstyle{every node}=[draw, inner sep=3pt]; 
\draw (0,3.7) node {\small$s_A - s_B$ is $D$-positive};
\draw (0,1.85) node {\begin{tabular}{l}\small$s_A - s_B$ is Schur-pos.\\[1.5ex] \small$s_A - s_B$ is $\QS$-positive\end{tabular}};
\draw (0,0) node {\small$s_A - s_B$ is $F$-positive};
\draw (5.1,2.15) node {\begin{tabular}{c}\small$\osupp{D}{A} \supseteq \osupp{D}{B}$  \\[1.5ex] \small$\ssupp{A} \supseteq \ssupp{B}$ \\[1.5ex] \small
$\osupp{\QS}{A} \supseteq \osupp{\QS}{B}$ \end{tabular}};
\draw (5.1,0) node {\small$\fsupp{A} \supseteq \fsupp{B}$};
\draw (10.9,0) node {\begin{tabular}{l}\small$\rowsk{k}{A} \domleq \rowsk{k}{B}\ \forall k$ \\[1.5ex]
\small$\colsk{\ell}{A} \domleq \colsk{\ell}{B}\ \forall \ell$ \\[1.5ex] 
\small$\rects{k}{\ell}{A} \leq \rects{k}{\ell}{B}\ \forall k,\ell$\end{tabular}};
\draw (0,-1.6) node {\small$s_A - s_B$ is $M$-positive};
\draw (5.1,- 1.6) node   {\small$\osupp{M}{A} \supseteq \osupp{M}{B}$};
\tikzstyle{every node}=[]; 
\draw (2.5,1.8) node {$\Rightarrow$};
\draw (2.5,0) node {$\Rightarrow$};
\draw (7.7,0) node {$\Rightarrow$};
\draw (2.5,-1.6) node {$\Rightarrow$};
\draw (0,0.8) node {$\Downarrow$};
\draw (0,2.9) node {$\Downarrow$};
\draw (0,-0.8) node {$\Downarrow$};
\draw (5.1,0.8) node {$\Downarrow$};
\draw (5.1,-0.8) node {$\Downarrow$};
\end{tikzpicture}
\caption{A summary of the implications of Section~\ref{sec:conclusion} for skew shapes $A$ and $B$.  With the exception of the rightmost implication, all the implications shown are known to be strict in the sense that the converse is false.}
\label{fig:moreimplications}
\end{center}
\end{figure}
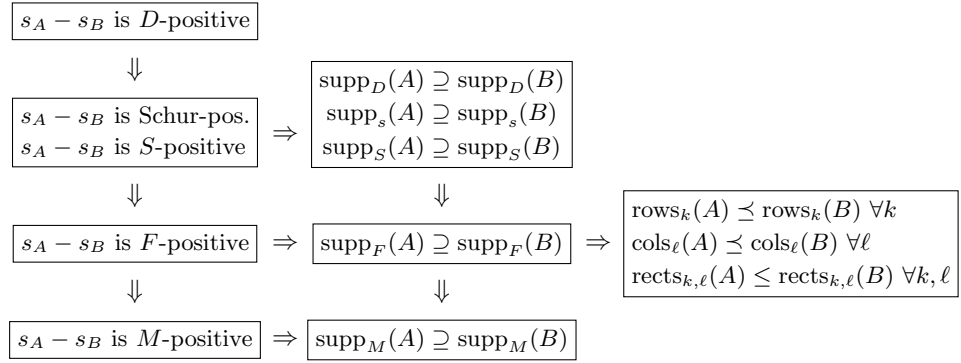

There are, of course, other known bases for quasisymmetric functions, such as those in \cite{BJR09,Luo08,Sta05}.  One possible first step to incorporating one of these other bases into our framework would be to determine the expansions of skew Schur functions in that new basis.

We now begin the derivation of the implications of Figure~\ref{fig:moreimplications}.  The implications at the bottom of the figure involving the $M$-basis are easy to see.  Indeed, the horizontal implication is by definition of support, and is strict since $\osupp{M}{3} \supseteq \osupp{M}{21}$ but $s_3 - s_{21}$ is not $M$-positive.  The vertical implications involving the $M$-basis are a consequence of any $F_\alpha$ being $M$-positive, as in~\eqref{equ:mtof}, and either implication can seen to be strict by comparing $s_{31/1}$ and $s_{211/1}$.  

The Schur functions have a very simple expansion in the $\QS$-basis:
\[
s_\lambda = \sum_\alpha \QS_\alpha,
\]
where the sum is over all compositions $\alpha$ that yield $\lambda$ when sorted into weakly decreasing order.  It follows that a symmetric function is Schur-positive if and only if it is $\QS$-positive and, similarly, Schur support containment for skew Schur functions is equivalent to $\QS$-support containment.

The derivation of the implications involving the $D$-basis requires a little more work.  
One interesting feature is that the Schur functions are not $D$-positive in general.  This means that there are two possible definitions of $\osupp{D}{A}$ for a skew shape $A$: we can either say that $\alpha$ is in the support if $D_\alpha$ appears with nonzero coefficient in the $D$-expansion of $s_A$, or we can insist that the coefficient be positive.  It turns out that it doesn't matter which convention we use in Figure~\ref{fig:moreimplications} or in any of the discussion that follows.

The expansion of $s_\lambda$ in the $D$-basis appears as \cite[Theorem~3.38]{BBSSZpr}: if $\lambda$ has $k$ parts, then
\[
s_\lambda = \sum_{\sigma} (-1)^\sigma D_{\lambda_{\sigma_1}+1-\sigma_1,\,\lambda_{\sigma_2}+2-\sigma_2,\, \ldots,\,  \lambda_{\sigma_k}+k-\sigma_k}\, ,
\]
where the sum is over all permutations $\sigma$ of $[k]$ such that $\lambda_{\sigma_i}+i-\sigma_i > 0$ for all $i \in [k]$.  Here $(-1)^\sigma$ denotes the sign of the permutation $\sigma$.   
Let us deduce some pertinent facts about this expansion.  Letting $\sigma$ be the identity permutation, we see that $D_\lambda$ appears with coefficient $+1$ in the $D$-expansion of $s_\lambda$.  Moreover, for any $\alpha$, it follows from \cite[Proposition~2.2]{BBSSZpr} that $D_\alpha$ appears with nonzero coefficient in the $D$-expansion of at most one $s_\lambda$.  In particular, $D_\lambda$ is the only term indexed by a partition that appears with nonzero coefficient in the $D$-expansion of $s_\lambda$.

If $s_A-s_B$ is $D$-positive then, in particular, the terms in $s_A-s_B$ of the form $D_\lambda$ with $\lambda$ a partition must all have nonnegative coefficients.  It then follows from the discussion of the previous paragraph that $s_A-s_B$ is Schur-positive.  An example that shows that this implication is strict is 
\[
s_{32/1}-s_{31} = s_{22} = D_{22} - D_{13}.
\]
Since each $\alpha$ appears in the $D$-support of at most one $s_\lambda$ and since $\lambda$ is in the $D$-support of $s_\lambda$, we deduce for skew shapes $A$ and $B$ that $\osupp{D}{A} \supseteq \osupp{D}{B}$ if and only if $\ssupp{A} \supseteq \ssupp{B}$.  
 
We can also quickly check that the implications in Figure~\ref{fig:moreimplications} inherited from Figure~\ref{fig:implications} are strict, with the possible exception of the rightmost arrow.  Skew shapes $A_1$ and $A_3$ from Table~\ref{tab:bases} show that $F$-positivity of $s_A-s_B$ does not imply Schur-positivity, and similarly for support containment.  Next, $A_1$ and $A_2$ from Example~\ref{exa:intro} satisfy $\fsupp{A_1} \supseteq \fsupp{A_2}$ but $s_{A_1} - s_{A_2}$ is not $F$-positive.  Finally, $s_{421/2} - s_{431/21} = -s_{32}$ is not Schur-positive, even though $\ssupp{421/2} \supseteq \ssupp{431/21}$.  This concludes our demonstration of all the implications of Figure~\ref{fig:moreimplications} and the desired strictness conditions.
 
But are there more implications that should be shown?  Let us impose the condition that $|A|=|B|$ in Figure~\ref{fig:moreimplications}, which does not change the figure or the substance of the implications.  Then one implication not shown in Figure~\ref{fig:moreimplications} that could be true is the implication of Conjecture~\ref{con:fsupport}.
Even if Conjecture~\ref{con:fsupport} is false, it could conceivably be the case that the row overlaps condition would imply containment of $M$-supports.  Apart from these exceptions, we can show that Figure~\ref{fig:moreimplications} is ``complete'' in the sense that all implications involving the various classes are implied by the implications shown.  For example, we will show that it is neither the case that $s_A-s_B$ being $M$-positive implies that $\fsupp{A} \supseteq \fsupp{B}$ nor vice versa.  To show completeness, there are four implications that we need to show are false; one can check that the absence of these four implications will imply the absence of any other conceivable implications within Figure~\ref{fig:moreimplications}.
\begin{itemize}
\item To see that $\ssupp{A} \supseteq \ssupp{B}$ does not imply that $s_A - s_B$ is $M$-positive,
take $A = 421/2$ and $B=431/21$ as at the end of the previous paragraph.  Then $s_A-s_B=-s_{32}$, which is not $M$-positive.  
\item To see that $s_A-s_B$ being $F$-positive does not imply that $\ssupp{A} \supseteq \ssupp{B}$, take $A=311/1$ and $B=22$ as in Table~\ref{tab:bases}.
\item To see that $s_A-s_B$ being $M$-positive does not imply that $\rowsk{k}{A} \domleq \rowsk{k}{B}$ for all $k$, take $A=3$ and $B=111$.  
\item To see that $\rowsk{k}{A} \domleq \rowsk{k}{B}$ for all $k$ does not imply that $s_A-s_B$ is $M$-positive, take $A=311/1$ and $B=32/1$, in which case $s_A-s_B=m_{1111}-m_{22}$
\end{itemize}

As a final remark, we have focussed on skew Schur functions because of the recent interest on relationships among them, as we described in the introduction, and because of their connection with the overlap partitions.  Of course, there may be other symmetric or quasisymmetric functions that would be worth comparing in the quasisymmetric setting.  Two natural prospects are the skew quasisymmetric Schur functions \cite{BLvW11}, which generalize the $\QS$-basis, and the skew dual immaculate functions \cite{BBSSZpr}.


\section*{Acknowledgements}

The author thanks Chris Berg, Nantel Bergeron, Christine Bessenrodt, Alejandro Morales, David Speyer and Mike Zabrocki for interesting discussions.  This research was performed while the author was on sabbatical at Trinity College Dublin, and he thanks the School of Mathematics for its hospitality.

\bibliographystyle{alpha}
\bibliography{F_support_nec_conditions}

\end{document}